\documentclass[reqno]{amsart}
\usepackage{amssymb,amscd,verbatim, amsthm,graphics, color,latexsym,amsmath,multicol,mathtools}
\usepackage{breqn}
\usepackage{fancybox}
\usepackage{tikz-cd}
\usepackage{mathrsfs}
\usepackage{longtable}
\usepackage[all]{xy}
\usepackage[top=2.7cm, bottom=2.7cm, left=2.7cm, right=2.7cm,marginparwidth=2cm]{geometry}
\usepackage{etoolbox}
\usepackage[colorinlistoftodos]{todonotes}
\usepackage[nogroupskip,nopostdot]{glossaries}
\usepackage{glossary-mcols}

\newtheorem{theorem}{Theorem}[section]

\newtheorem{assumption}[theorem]{Assumption}
\newtheorem{corollary}[theorem]{Corollary}
\newtheorem{lemma}[theorem]{Lemma}
\newtheorem{proposition}[theorem]{Proposition}
\newtheorem{definition}[theorem]{Definition}
\newtheorem{remark}[theorem]{Remark}
\newtheorem{example}[theorem]{Example}
\newtheorem{thmintro}{Theorem}

\newcommand{\lpi}{\langle}
\newcommand{\rpi}{\rangle}

\newcommand{\fa}{\mathfrak{a}}
\newcommand{\fb}{\mathfrak{b}}

\newcommand{\fg}{\mathfrak{g}}
\newcommand{\fh}{\mathfrak{h}}

\newcommand{\fl}{\mathfrak{l}}
\newcommand{\fm}{\mathfrak{m}}
\newcommand{\fn}{\mathfrak{n}}
\newcommand{\fo}{\mathfrak{o}}
\newcommand{\fp}{\mathfrak{p}}

\newcommand{\fs}{\mathfrak{s}}
\newcommand{\ft}{\mathfrak{t}}
\newcommand{\fu}{\mathfrak{u}}

\newcommand{\bba}{\mathbb{A}}

\newcommand{\bbc}{\mathbb{C}}

\newcommand{\bbe}{\mathbb{E}}

\newcommand{\bbn}{\mathbb{N}}

\newcommand{\bbr}{\mathbb{R}}
\newcommand{\bbs}{\mathbb{S}}
\newcommand{\bbt}{\mathbb{T}}
\newcommand{\bbu}{\mathbb{U}}
\newcommand{\bbv}{\mathbb{V}}

\newcommand{\bbz}{\mathbb{Z}}
\newcommand{\Ca}{\mathcal{A}}
\newcommand{\Cb}{\mathcal{B}}
\newcommand{\Cc}{\mathcal{C}}

\newcommand{\Ce}{\mathcal{E}}
\newcommand{\Cf}{\mathcal{F}}
\newcommand{\Cg}{\mathcal{G}}
\newcommand{\Ch}{\mathcal{H}}

\newcommand{\Cm}{\mathcal{M}}

\newcommand{\Co}{\mathcal{O}}
\newcommand{\Cp}{\mathcal{P}}

\newcommand{\Cr}{\mathcal{R}}

\newcommand{\Cu}{\mathcal{U}}
\newcommand{\Cv}{\mathcal{V}}
\newcommand{\Cw}{\mathcal{W}}
\newcommand{\Cx}{\mathcal{X}}
\newcommand{\Cy}{\mathcal{Y}}

\newcommand{\bt}{\mathbf{t}}

\newcommand{\SM}{\mathscr{M}}

\newcommand{\SR}{\mathscr{R}}

\newcommand{\lcont}{\;\llcorner\;}
\newcommand{\rcont}{\;\lrcorner\;}

\newcommand{\cpv}{\begin{center}\begin{minipage}[t]{14cm}\small{\it Proof.} }
\newcommand{\fpv}{\fim\end{minipage}\end{center}}
\newcommand{\fim}{\hfill $\Box$}

\newcommand{\cha}{\mathsf{H}_{c}}

\DeclareMathOperator{\sgn}{sgn}

\DeclareMathOperator{\im}{im}

\DeclareMathOperator{\Hom}{Hom}

\DeclareMathOperator{\Cent}{Cent}

\DeclareMathOperator{\ad}{ad}

\DeclareMathOperator{\End}{End}

\DeclareMathOperator{\triv}{triv}

\edef\epi{\expandafter[twoheadrightarrow]}
\edef\mono{\expandafter[hookrightarrow]}
\def\xar{\expandafter\arrow}

\numberwithin{equation}{section}
\patchcmd{\subsection}{\bfseries}{\itshape}{}{}
\makenoidxglossaries
\newglossaryentry{EucStruc}{ name={\ensuremath{(\cdot,\cdot)}},sort={bracket euclidean},description={} } %
\newglossaryentry{BilPair}{ name={\ensuremath{\lpi\cdot,\cdot\rpi}},sort={bracket pairing},description={} } %
\newglossaryentry{Rts}{ name={\ensuremath{R}},sort={roots},description={} } %
\newglossaryentry{Rp}{ name={\ensuremath{R_+}},sort={roots positive},description={} } %
\newglossaryentry{Weyl}{ name={\ensuremath{W}},sort={Weyl group},description={} } %
\newglossaryentry{RCA}{ name={\ensuremath{\cha}},sort={H Cherednik algebra},description={} } %
\newglossaryentry{IrrW}{ name={\ensuremath{\widehat{W}}},sort={Weyl group irreps},description={} } %
\newglossaryentry{Std}{ name={\ensuremath{M_c(\tau)}},sort={M-standard},description={} } %
\newglossaryentry{Nc}{ name={\ensuremath{N_c(\tau)}},sort={N-scalar},description={} } %
\newglossaryentry{Star}{ name={\ensuremath{\ast}},sort={star},description={} } %
\newglossaryentry{Cspace}{ name={\ensuremath{\mathscr{C}}},sort={c},description={} } %
\newglossaryentry{UnitNgh}{ name={\ensuremath{U(\tau)}},sort={Unitary locus},description={} } %
\newglossaryentry{Cliff}{ name={\ensuremath{\Cc}},sort={clifford},description={} } %
\newglossaryentry{Spin}{ name={\ensuremath{\bbs}},sort={spin module},description={} } %
\newglossaryentry{UNilp}{ name={\ensuremath{\fu^-}},sort={u-minus-subalgebra},description={} } %
\newglossaryentry{NNilp}{ name={\ensuremath{\fn^-}},sort={n-minus-subalgebra},description={} } %
\newglossaryentry{Lightning}{ name={\ensuremath{L(+\infty,\delta_2)}},sort={lightning},description={} } %
\newglossaryentry{Tau}{ name={\ensuremath{\tau_\alpha}},sort={tau},description={} } %
\newglossaryentry{Omega}{ name={\ensuremath{\Omega_c}},sort={omega},description={} } %
\newglossaryentry{Basy}{ name={\ensuremath{\{y_j\}}},sort={y-basis},description={} } %
\newglossaryentry{Basx}{ name={\ensuremath{\{x_i\}}},sort={x-Basis 2},description={} } %
\newglossaryentry{Diag}{ name={\ensuremath{\rho}},sort={rho-diagonal embedding},description={} } %
\newglossaryentry{SPO}{ name={\ensuremath{\fg}},sort={g-spo},description={} }  %
\newglossaryentry{Cliffn}{ name={\ensuremath{\Cc_n}},sort={Clifford algebra degree n},description={} } %
\newglossaryentry{Kmod}{ name={\ensuremath{K_c(\tau)}},sort={K-standard-spinor},description={} } %
\newglossaryentry{HermStruc}{ name={\ensuremath{\lpi\cdot|\cdot\rpi}},sort={bracket hermitian},description={} } %
\newglossaryentry{gKmod}{ name={\ensuremath{K^l_m}},sort={K-standard-spinor 2},description={} } %
\newglossaryentry{gKmodIso}{ name={\ensuremath{K^l_m(\sigma)}},sort={K-standard-spinor 3},description={} } %
\newglossaryentry{Lambda}{ name={\ensuremath{\lambda_{m,l,\sigma}}},sort={lambda},description={} } %
\newglossaryentry{Monog}{ name={\ensuremath{\Cm_c(\tau)}},sort={monogenics},description={} } %
\newglossaryentry{Harm}{ name={\ensuremath{\Ch_c(\tau)}},sort={harmonics},description={} } %
\newglossaryentry{AHarm}{ name={\ensuremath{\Ca_c(\tau)}},sort={A-harmonics},description={} } %
\newglossaryentry{GL}{ name={\ensuremath{\fl}},sort={l},description={} } %
\newglossaryentry{AHarmHom}{ name={\ensuremath{\Ca(n)}},sort={A-harmonics2},description={} } %
\newglossaryentry{AHarmHomIso}{ name={\ensuremath{\Ca(n,\sigma)}},sort={A-harmonics3},description={} } %
\newglossaryentry{Br}{ name={\ensuremath{B(r)}},sort={B},description={} } %
\newglossaryentry{P}{ name={\ensuremath{\Cp(1|1)}},sort={P},description={} } %
\newglossaryentry{Hr}{ name={\ensuremath{H(r)}},sort={H},description={} } %
\newglossaryentry{SL}{ name={\ensuremath{\fs}},sort={s},description={} }%
\newglossaryentry{Xgen}{ name={\ensuremath{\Cx_{ij}}},sort={Xij},description={} }%
\newglossaryentry{SO}{ name={\ensuremath{\fa}},sort={a},description={} } %
\newglossaryentry{CentSL}{ name={\ensuremath{\Cu_c}},sort={U},description={} } %
\newglossaryentry{WeylAlg}{ name={\ensuremath{\Cw}},sort={Weyl algebra},description={} } %
\newglossaryentry{ClassCentSO}{ name={\ensuremath{A(\fa)}},sort={Aa},description={} } %
\newglossaryentry{WeylGen}{ name={\ensuremath{\Cw[\bt]}},sort={Weyl algebra generic},description={} } %
\newglossaryentry{RCAGen}{ name={\ensuremath{\cha[\bt]}},sort={H Cherednik algebra generic},description={} }%
\newglossaryentry{Xgen2}{ name={\ensuremath{\Cx_{xy}}},sort={Xij2},description={} } %
\newglossaryentry{CentSPO}{ name={\ensuremath{\Cr_c}},sort={R},description={} } %
\newglossaryentry{CentreDiag}{ name={\ensuremath{Z_W}},sort={Zentre weyl group},description={} } %
\newglossaryentry{diff}{ name={\ensuremath{d}},sort={Dirac differential},description={} } %
\newglossaryentry{diffW}{ name={\ensuremath{d_W}},sort={Dirac differential W-invariant},description={} } %
\newglossaryentry{Xdiag}{ name={\ensuremath{\Cx^\Delta_{ij}}},sort={Xij3},description={} } %
\newglossaryentry{SOdiag}{ name={\ensuremath{\fa^\Delta}},sort={a-diagonal},description={} } %
\newglossaryentry{XdiagF}{ name={\ensuremath{\Cf_{ij}}},sort={F-Xij},description={} } %
\newglossaryentry{Valg}{ name={\ensuremath{\Cv_c}},sort={V-algebra},description={} } %
\newglossaryentry{XdiagG}{ name={\ensuremath{\Cg_{ij}}},sort={G-Xij},description={} }  %
\newglossaryentry{Pairs}{ name={\ensuremath{\Pi}},sort={pairs},description={} } %
\newglossaryentry{orderM}{ name={\ensuremath{\ll}},sort={order1},description={} } %
\newglossaryentry{Multiindices}{ name={\ensuremath{\SM}},sort={multiindices},description={} } %
\newglossaryentry{Multiindicesp}{ name={\ensuremath{\SM_p}},sort={multiindices-p},description={} } %
\newglossaryentry{orderM2}{ name={\ensuremath{\prec}},sort={order2},description={} } %
\newglossaryentry{Matij}{ name={\ensuremath{\Cm_{ij}}},sort={matrix entry},description={} } %
\newglossaryentry{Matmu}{ name={\ensuremath{\Cm_{\mu}}},sort={matrix entry},description={} } %
\newglossaryentry{Basisij}{ name={\ensuremath{\Cb_{ij}}},sort={basis matrix entry},description={} } %
\newglossaryentry{Basismu}{ name={\ensuremath{\Cb_{\mu}}},sort={basis matrix entry},description={} } %
\newglossaryentry{orderR}{ name={\ensuremath{\lll}},sort={order3},description={} }
\newglossaryentry{MultiindicesCl}{ name={\ensuremath{\SR}},sort={r-multiindices},description={} }
\newglossaryentry{MultiindicesClp}{ name={\ensuremath{\SR_p}},sort={r-multiindices-p},description={} }
\makeatletter
\newglossarystyle{TwoColListSymb}{%
    \setglossarystyle{alttree}%
    {%
        \begin{multicols}{2}%
            \def\@gls@prevlevel{-1}%
        }%
        {\par\end{multicols}}%
}
\makeatother

\begin{document}

\title{The Dunkl-Cherednik deformation of a Howe duality}

\author
{Dan Ciubotaru}
\author
{Marcelo De Martino}
        \address{Mathematical Institute, University of Oxford, Oxford OX2 6GG, UK}
        \email{dan.ciubotaru@maths.ox.ac.uk, marcelo.demartino@maths.ox.ac.uk}

\begin{abstract}
We consider the deformed versions of the classical Howe dual pairs $(O(r,\bbc),\fs\fl(2,\bbc))$ and $(O(r,\bbc),\fs\fp\fo(2|2,\bbc))$ in the context of a rational Cherednik algebra $\cha=\cha(W,\fh)$ associated to a finite Coxeter group $W$
at the parameters $c$ and $t=1$. For the first pair, we compute the centraliser of the well-known copy of $\fs\cong\fs\fl(2,\bbc)$ inside $\cha$. For the second pair, we show that the classical copy of $\fg\cong\fs\fp\fo(2|2,\bbc)$
inside the Weyl-Clifford algebra $\Cw\otimes\Cc$ deforms to a Lie superalgebra inside $\cha\otimes\Cc$ and compute its centraliser algebra. 
For a generic parameter $c$ such that the standard $\cha$-module is unitary, we compute the joint $((\cha)^\fs,\fs)$- and $((\cha\otimes\Cc)^\fg,\fg)$-decompositions of the relevant modules.
\end{abstract}

\thanks{This work was supported in part by the Engineering and Physical Sciences Research Council grant [EP/N033922/1] (2016).}


\maketitle

\section{Introduction}
\subsection{}In his influential paper \cite{Ho}, Howe described a uniform formulation of the First Fundamental Theory of Classical Invariants in terms of the Weyl algebra (or the Weyl-Clifford algebra). 
More precisely, let $G$ be a classical Lie group (over $\bbc$) and let $V$ be its standard module. Consider the 
complex vector spaces $U_0$ and $U_1$ defined as finite direct sums of copies of $V$ and $V^*$ and let $\Ca(U_0,U_1) = S(U_0)\otimes\wedge(U_1)$. Let $\Cw\Cc(U_0,U_1)$ denote 
the unital associative subalgebra of $\End(\Ca(U_0,U_1))$ generated by multiplication by elements in $U_0$ or $U_1$ and differentiation by elements in $U_0^*$ or $U_1^*$. This algebra 
is isomorphic to  $\Cw(U_0)\otimes\Cc(\tilde U_1)$, where $\Cw(U_0)$ is the Weyl algebra of partial differential operators on $U_0$ with polynomial coefficients and $\Cc(\tilde U_1)$ is the 
Clifford algebra associated to the complex vector space $\tilde U_1 = U_1\oplus U_1^*$ and the symmetric bilinear form given as the extension of the natural bilinear pairing $U_1\times U_1^* \to \bbc$. 
The results of \cite{Ho}, now commonly referred to as {\it Howe duality}, give the multiplicity-free joint decomposition of $\Ca(U_0,U_1)$ for the pair $(G,\Gamma')$, where $\Gamma'$ is a Lie algebra or superalgebra
that generates $\Cw\Cc(U_0,U_1)^G$, the algebra of $G$-invariants inside the Weyl-Clifford algebra. 

A first example occurs for $G = O(r,\bbc)$, $U_0 = V^*\cong \bbc^r$
and $U_1 = \{0\}$. In this case, $\Ca(U_0,U_1)$ is the algebra $P(V)\cong \bbc[z_1,\ldots,z_r]$ of polynomial functions on $V$. The Lie algebra $\fs$ spanned by the Laplacian $\Delta$, the norm-square operator
$|z|^2 = \sum_jz_j^2$ and $4h:=[\Delta,|z|^2] = 4\bbe +2r$, where $\bbe = \sum_j z_j\partial_{z_j}$ is the Euler's degree operator, generate a copy of $\fs\cong\fs\fl(2,\bbc)$ which is the appropriate Lie algebra
that generates the $O(r,\bbc)$-invariants inside the Weyl algebra. The Howe duality obtained is the multiplicity-free joint $(O(r,\bbc),\fs\fl(2,\bbc))$-decomposition
\begin{equation}\label{e:HoweDual1}
P(V)= \bigoplus_{m\in \bbz_{\geq 0}}\Ch_m\boxtimes L_{\fs\fl(2)}(m + r/2),
\end{equation}
where $\boxtimes$ denotes the external tensor product, $\Ch_m$ is the space of harmonic polynomials of degree $m$, which is a finite-dimensional irreducible $O(r,\bbc)$-module, and $L_{\fs\fl(2)}(m + r/2)$ is the irreducible 
infinite-dimensional lowest weight $\fs\fl(2,\bbc)$-module of lowest weight $m + r/2 \in \bbc$. 

Still for $G=O(r,\bbc)$, a second example, which will be very important to the present work and is explained in \cite{CW} together 
with many other explicit examples of Howe dualities for Lie superalgebras, is for the case $U_0 = V^* = U_1$. Here, $\Ca(U_0,U_1)$ reduces to the space $\Omega(V)=\oplus_{m,l\in\bbz_{\geq 0}}\Omega_m^l(V)$ of polynomial differential forms on $V$, where $\Omega_m^l(V) = P_m(V)\otimes \wedge^lV^*$, and the algebra $\fg$ that generates the invariants inside the appropriate Weyl-Clifford algebra is isomorphic to the Lie superalgebra $\fg \cong \fs\fp\fo(2|2,\bbc)$. Below, 
in Subsection \ref{s:spo}, we shall recall the basic structure theory of this algebra and in Section \ref{s:LWTheory} some of its theory of lowest weight representations.
For the purposes of this introduction, it suffices to mention that in the decomposition $\fg = \fg_{\bar 0}\oplus \fg_{\bar 1}$, the even subalgebra is isomorphic to $\fs\fp(2,\bbc)\oplus \fs\fo(2,\bbc)\cong\fg\fl(2,\bbc)$. It contains the copy
of $\fs\cong \fs\fl(2,\bbc)$ that appears in the decomposition (\ref{e:HoweDual1}).  The Cartan subalgebra of $\fg$ is two-dimensional, spanned by elements $h,z$, where $h$ is the semisimple element of $\fs\cong\fs\fl(2,\bbc)$
and $z$ spans the centre of $\fg_{\bar 0}\cong \fg\fl(2,\bbc)$. Thus, a weight for a $\fg$-representation will be represented by a pair $(\mu,\nu)\in\bbc^2$. We shall fix a choice of negative root-vectors (see (\ref{e:nilpot}) and (\ref{e:spodefs})) 
amongst which lies the Laplacian.
To describe the Howe-duality obtained in this case, let $B(r)$ denotes the set $\{(0,0),(1,r-1)\}\cup\{(m,l)\mid 1\leq m, 0\leq l \leq r-2\}\subseteq \bbz_{\geq 0}\times \bbz_{\geq 0}$.
The statement is the  multiplicity-free joint $(O(r,\bbc),\fs\fp\fo(2|2,\bbc))$-decomposition (see \cite[Theorem 5.39]{CW}; their decomposition is labelled by hook-partitions and we will reinterpret the set $B(r)$ in that setting in (\ref{e:hpartition}))
\begin{equation}\label{e:HoweDual2}
\Omega(V)= \bigoplus_{(m,l)\in B(r)}\Cm_{m}^{l}\boxtimes L_{\fs\fp\fo(2|2,\bbc)}(m+r/2,l-r/2),
\end{equation}
where $L_{\fs\fp\fo(2|2,\bbc)}(m+r/2,l-r/2)$ denotes an irreducible infinite-dimension lowest weight $\fg$-representation  and $\Cm_{m}^{l}\subseteq\Omega_m^l(V)$ denotes the space space of lowest weight vectors for the action of $\fg$, which is a 
finite-dimensional irreducible $O(r,\bbc)$-module.

\subsection{}In this paper, instead of directional derivatives and the Weyl algebra, we shall consider the differential-difference operators introduced by Dunkl \cite{Du} and the rational Cherednik algebra \cite[Section 4]{EG}. 
Following the customary notations, let $W$ denote a real reflection group acting on a Euclidean vector space of dimension $r$ and denote by $\fh$ its (complexified) reflection representation. 
We denote by $\cha = \cha(W,\fh)$ the rational Cherednik algebra at the parameter $c$ (and $t=1$, in \cite{EG}). 
For each irreducible representation $\tau$ of $W$, we let $M_c(\tau)$ be the corresponding standard module. 
In Section \ref{s:defs} we shall recall the definition of these objects. Our main goals are to
describe the decompositions analogous to (\ref{e:HoweDual1}) and (\ref{e:HoweDual2}) in the context of the rational Cherednik algebra $\cha$.

Starting with the pair that yields the decomposition in (\ref{e:HoweDual1}), it has been known \cite{He} that the Dunkl-version of the $\fs\cong\fs\fl(2,\bbc)$ copy inside the Weyl algebra survives the deformation 
and is also present in the rational Cherednik algebra $\cha$, for all values of the parameter $c$. One of the issues encountered when passing to the deformed version is that the full action of the orthogonal group is lost. 
However, we show that an algebra closely related to the orthogonal Lie algebra does exist for all parameters $c$. Fix an 
orthonormal basis of $\fh$ and let $\Cx_{ij}$, with $1\leq i<j\leq r$ denote the differential-difference version of the classical realisation of the basis of the orthogonal Lie algebra $\fs\fo(r,\bbc)$ inside the Weyl algebra
(see (\ref{e:o})). The first result we prove is the following: 

\begin{thmintro}
Let $\Cu_c$ denote the unital, associative subalgebra of $\cha$ generated by $\bbc W$ and the elements $\Cx_{ij}$ with $1 \leq i<j \leq r$. Then $\Cu_c$ is the centraliser algebra of $\fs\cong \fs\fl(2,\bbc)$ in $\cha$.
\end{thmintro}

The algebra $\Cu_c$ has been studied in \cite{FeHa}, where it was called the Dunkl angular momenta algebra. It is a non-homogeneous quadratic algebra of PBW-type in the sense of \cite{BG}.  We briefly discuss the relations between its generators 
$\Cx_{ij}$ in Section \ref{s:6.2}.

The idea to consider deformed versions of Howe duality for certain dual pairs has, of course, been explored before. In \cite{BSO}, for example, the notion of generalised Fock space was introduced in order to study
the classical Segal-Bargmann transform in the context of Dunkl-operators. The Dunkl-version of the $\fs\fl(2)$-triple mentioned above acts on this Fock space as annihilation, creation and number operators.
Among other results \cite[Section 5]{BSO}, Ben-Sa\"{i}d and \O rsted exhibited a $(W,\fs)$-decomposition of the polynomial space analogous to (\ref{e:HoweDual1}), as well as precise formulas for the projection onto the spaces
of Dunkl-harmonic polynomials and a Hecke-type formula for the Dunkl-transform. Their decomposition theorem, thus, yields a deformed version of the $(O(r,\bbc),\fs\fl(2,\bbc))$-duality. 
With respect to the dual pair $(\Cu_c,\fs)$ we find the following refinement of their decomposition:

\begin{thmintro}
Let $\tau$ be an irreducible $W$ representation and fix a generic parameter $c$ (as in Assumption \ref{a:assumption}) for which the standard module $M_c(\tau)$ is unitary. Then, as a $(\Cu_c,\fs)$-module,
we have a multiplicity-free decomposition 
\[
M_c(\tau) = \bigoplus_{m\in\bbz_{\geq 0}} \Ch_c(\tau)_m  \boxtimes L_{\fs\fl(2)}(m + r/2 - N_c(\tau)),
\]
where $N_c(\tau)$ is the scalar by which the central element $\Omega_c=\sum_{\alpha>0}c_\alpha s_\alpha$ in $\bbc W$ acts on $\tau$ and $\Ch_c(\tau)_m$ is the space of the Dunkl-harmonic elements of degree $m$. 
In particular, each $\Ch_c(\tau)_m$ is a finite-dimensional unitary simple $\Cu_c$-module.
\end{thmintro}

Before we describe our results obtained for the deformation of the second decomposition (\ref{e:HoweDual2}), we mention that the related pair $(\textup{Pin}(r),\fs\fp\fo(2|1))$ was studied before in, for example, \cite{DBOSS} and \cite{OSS},
where, explicit formulas for the projection onto the kernel of the Dunkl-Dirac operator are given and Fourier-type transforms associated to the Dunkl-Dirac operator are studied. Subsequently, \cite{DBOVJ} also studies the centraliser algebra of the relevant superalgebra that appears and its relations with the higher-rank Bannai-Ito algebra. (See Section 4 of {\it loc. cit.} and the references therein.)  

The case that we consider in this paper, $(O(r,\bbc),\fs\fp\fo(2|2,\bbc))$,
has two additional difficulties in comparison to the cases mentioned so far. The first one is that finite-dimensional $\fs\fp\fo(2|2,\bbc)$-representations are not necessarily completely reducible, like they are for $\fs\fl(2,\bbc)$ and $\fs\fp\fo(2|1,\bbc)$.
The second and main difference is that, when considering the deformed version of $\fg = \fs\fp\fo(2|2,\bbc)$, the central element of the even subalgebra $\fg_{\bar 0}$ contains the extra term $\Omega_c$ (see (\ref{e:omega})) which acts diagonally on both the
polynomial part and the exterior algebra part and depends, in an essential way, on the parameter function $c$. 
The element $\Omega_c$ lies in the centre of the $\bbc W$; it has already been present in the computations of the square of the Dirac operator \cite{Ci} in the context
of Drinfeld algebras. 

Using the Dirac operator, we define elements $\Cf_{ij}$ (see (\ref{e:sodeformed1}) and (\ref{e:sodeformed2})) which are closely
related to the classical diagonal realisation of the  orthogonal Lie algebra (see \cite[Section 5.3.4]{CW}) inside $\cha\otimes\Cc$ (here $\Cc$ is the Clifford algebra of $\fh\oplus\fh^*$). We prove

\begin{thmintro}
Let $\Cv_c$ denote the unital associative subalgebra of $\cha\otimes\Cc$ generated by the diagonal copy of $\bbc W$ and the elements $\Cf_{ij}$ with $1 \leq i<j \leq r$.  Let $Z_W$ denote the centre of $\bbc W$ and let $
\mathcal R_c$ denote the centraliser algebra
of $\fg\cong\fs\fp\fo(2|2,\bbc)$ in $\cha\otimes\Cc$.
Then, for all $c$, $\Cv_c^{Z_W}\subseteq \mathcal R_c$, and when $c$ is generic, $\Cv_c^{Z_W}= \mathcal R_c$.
\end{thmintro}
We remark that when $c=0$, $\mathcal R_0=\Cv_0\cong A(\mathfrak{so}(\mathfrak h)^\Delta)\# W$, where $A(\mathfrak{so}(\mathfrak h)^\Delta)$ denotes the associative subalgebra of $\Cw\otimes \Cc$ generated by the Lie algebra $\mathfrak{so}(\mathfrak h)$ embedded diagonally in $\mathcal W\otimes \mathcal C$.

\smallskip

Just as for the deformed version of the decomposition (\ref{e:HoweDual1}), fix an irreducible representation $\tau$ of $W$ and a generic parameter $c$ (as in Assumption \ref{a:assumption}) for which the standard 
module $M_c(\tau)$ is unitary and consider the module $K_c(\tau) = M_c(\tau)\otimes\wedge\fh^*$.

\begin{thmintro}\label{t:ospduality}
As an $(\Cr_c,\fg)$-module, the space $K_c(\tau)$ decomposes as 
\[
K_c(\tau) =\bigoplus_{\sigma\in \widehat{W}} \bigoplus_{(m,l)\in B(r)} \Cm_c(\tau,\sigma)_{m}^{l}  \boxtimes L_{\fs\fp\fo(2|2,\bbc)}(m+r/2-N_c(\tau),l-r/2+N_c(\sigma)).
\]
Here, $\Cm_c(\tau,\sigma)_{m}^{l}$ is the $\sigma$-isotypic component of the space of lowest weight vectors at the bidegree $(m,l)$. 
Each $\Cm_c(\tau,\sigma)_{m}^{l}$ is a finite dimensional unitary $\Cr_c$-module.
\end{thmintro}

 In the last Section, we discuss the centralizer algebras of the classical, non-deformed versions
of the realizations of $\fs\fl(2,\bbc)$ and $\fs\fp\fo(2|2,\bbc)$ in the Weyl and Weyl-Clifford setting. Although the computation of these centralizer algebras seems to be well-known, we provide proofs for convenience.

\subsection{Acknowledgements}
We thank H. De Bie, R. Oste, and K. McGerty for the fruitful discussions we had while preparing this work. We are also thankful to the anonymous referees for the reference \cite{FeHa} and for valuable feedback that improved the manuscript.

\section{Initial definitions and notations}\label{s:defs}

\subsection{The rational Cherednik algebra}
Let $E = \bbr^r$ be a Euclidean space and identify $E$ and $E^*$ by means of the Euclidean structure.
We shall denote by $\gls{EucStruc}$ the Euclidean structures on both spaces.

Let $\gls{Rts}\subseteq E^*$ be a root system, with $\gls{Rp}$ a fixed set of positive roots. We will typically use the notation $\alpha > 0$ instead of $\alpha\in R_+$.
Denote by $\gls{Weyl}$ the Weyl group of $R$ and fix a parameter function $c:R\to\bbc$, denoted $c(\alpha)=c_\alpha$ which satisfies $c_{w\alpha} = c_\alpha$, for all $w\in W$ and $\alpha\in R$. To each $\alpha\in R$, 
denote by $s_\alpha$ the reflection with respect to the hyperplane orthogonal to $\alpha$. We extend this to a linear operator
on $\bbc[E]$. Denote by $\Delta_\alpha$ the corresponding divided-difference operator
\[
\Delta_\alpha = \tfrac{1}{\alpha}(1 - s_\alpha)\in\End(\bbc[E]).
\]
 For each $\eta\in E$, we let $T_\eta\in\End(\bbc[E])$ denote the Dunkl operator, defined by
 \[
 T_\eta(p) = \partial_\eta(p) - \sum_{\alpha>0}c_\alpha ( \alpha,\eta ) \Delta_\alpha(p),
 \] 
for all $p\in\bbc[E]$. For every $\xi\in E^*$, let $M_\xi$ denote the multiplication operator of $\bbc[E]$. The  associative subalgebra of $\End(\bbc[E])$ generated $\{M_\xi,T_\eta,w\mid\xi\in E^*,\eta\in E,w\in W\}$ is a  realisation
of the so-called rational Cherednik algebra (see \cite[Section 4]{EG}). Formally, let $\fh = E_\bbc$, denote by $\bbt(\fh\oplus\fh^*)$ the tensor algebra of the complex vector space $\fh\oplus\fh^*$ and by $\gls{BilPair}$ the bilinear 
pairing $\fh^*\times\fh\to\bbc$. 
Whenever $A$ is a $\bbc$-algebra and $G$ is a finite group acting on $A$,  we shall denote by $A\#G$ the skew-group algebra, which satisfies the relations $gag^{-1} = g(a)$, for all $g\in G$ and $a\in A$.

\begin{definition}\label{d:RCA}
The \emph{rational Cherednik algebra} $\gls{RCA}=\cha(W,\fh)$ is the unital, associative algebra over $\bbc$ given as the quotient of $\bbt(\fh\oplus \fh^*)\# W$ modulo the relations
\begin{equation}\label{e:relations}
\begin{aligned}
&[y,y'] = 0 = [x,x'] ,& y,y'\in \fh,x,x'\in \fh^*,\\
&[y,x] = \lpi x,y \rpi - \sum_{\alpha>0} c_\alpha\lpi \alpha,y \rpi\lpi x,\alpha^\vee\rpi s_\alpha,& y\in \fh,x\in \fh^*.
\end{aligned}
\end{equation}
\end{definition}
In what follows, we shall denote the rational Cherednik algebra simply by $\cha$, leaving the data $(W,\fh)$ implicit. As a vector space, this algebra satisfies the PBW property $\cha = \bbc[\fh]\otimes\bbc W\otimes\bbc[\fh^*]$, and this triangular decomposition
leads to the notion of category $\Co_c$. We refer the reader to \cite{GGOR} for the details on this theory. Denote by
$\gls{IrrW}$ the set of equivalence classes of irreducible $\bbc W$-modules.
For any $\tau\in\widehat{W}$ with representation space $V(\tau)$, we shall denote by $\gls{Std}\in \Co_c$ the corresponding standard module $M_c(\tau)=\cha\otimes_{W\# \bbc[\fh^*]}V(\tau)$, where $\fh$ acts on $V(\tau)$ by zero. We shall denote by $\gls{Nc}\in\bbc$, the scalar by which the central element 
$\sum_{\alpha>0}c_\alpha s_\alpha\in\bbc W$ acts on the
irreducible module $V(\tau)$.

Extend the Euclidean structure on $E$ to a positive-definite Hermitian structure on $\fh=E_\bbc$ (and on $\fh^*=E^*_\bbc$), which we shall still denote by $(\cdot,\cdot)$ and make the convention that this pairing is linear in the first variable. This defines 
an anti-linear 
isomorphism $\gls{Star}:\fh\to\fh^*$ defined by $y_1^*(y_2) = (y_2,y_1)$ for all $y_1,y_2\in\fh$. We denote the inverse also by $\ast$ and extend it to an anti-linear anti-involution on $\cha$, by setting $s_{\alpha}^*=s_\alpha$, for all $\alpha\in R_+$. 

Fix a $W$-invariant unitary structure on $V(\tau)$. Denote by $\beta_{c,\tau}$ the unique $W$-invariant Hermitian form on $M_c(\tau)$ that coincides with the Hermitian structure of $V(\tau)$ in degree zero and satisfies the contravariance property
\[\beta_{c,\tau}( yz,z' ) = \beta_{c,\tau}( z,y^*z' ),\] 
for all $y\in\fh$ and $z,z'\in M_c(\tau)$ (see \cite[Proposition 2.2]{ES}, \cite[Section 2.4]{DO} and also \cite[Theorem 3.7]{BSO}). It follows that $(M_c(\tau),\beta_{c,\tau})$ is a $*$-Hermitian $\cha$-module. 

\begin{lemma}\label{l:unitarity}
Denote by $\gls{Cspace}$ the $\bbr$-vector space of real-valued parameter functions. Assume $c\in \mathscr{C}$ and fix
$\tau\in \widehat{W}$.
\begin{enumerate}
\item[(a)] For $c=0$, the standard module $M_0(\tau)$ is unitarisable and irreducible. 
\item[(b)] There exists a neighbourhood $\gls{UnitNgh}$ of $0\in\mathscr{C}$ such that $M_c(\tau)$ is unitarisable and irreducible for all $c\in U(\tau)$. 
\item[(c)] Let $\tau = \triv$ and for each $\sigma\in \widehat{W}$, let $m(\sigma)$ denote the minimal homogeneous degree in 
which $\sigma$ occurs in $\bbc[\fh]$. Then,  $M_c(\triv)$ is unitarisable and irreducible if and only if $c$ is in the convex region of $\mathscr{C}$
determined by the inequalities $(m(\sigma) - (N_c(\triv) - N_c(\sigma))>0$, for $\sigma\in \widehat{W}\setminus\{\triv\}$.
\end{enumerate}
\end{lemma}

\begin{proof}
Items (a) and (b) were discussed in \cite[Proposition 3.1]{ES}. Item (c) is \cite[Proposition 2.24]{DO}. Because we use different 
notations, for convenience, we sketch the argument. Denote only by $\beta_c$ the contravariant form for $\tau=\triv.$
We shall show that $\beta_c$ is positive-definite if and only if the condition in item (c) is fulfilled.
 
Fix an orthonormal basis $\{y_j\}\subseteq E\subseteq \fh$ and let $\{x_j\}\subseteq \fh^*$ be a dual basis. Note that
this implies that $x_j^*=y_j$, for all $1\leq j\leq \dim E$.  If $p\in M_c(\triv)$ is a homogeneous polynomial of degree $m$ contained in the $\sigma$-isotypic component 
then it is standard that the element $E(c)=\sum_jx_jy_j\in\cha$ acts on $p$ by the scalar 
$m - (N_c(\triv)-N_c(\sigma))$. Note that if $p$ is $W$-invariant, this reduces to the polynomial degree.
Further, assume by induction on the polynomial degree $m$ that $\beta_c(q,q)>0$, whenever $0\neq q$ and $q$ is homogeneous of degree $l<m$. If $0\neq p$ is homogeneous of degree $m\geq m(\sigma)$ and is in the 
$\sigma$-isotypic component, then
\begin{equation}\label{e:posdefeq}
(m - (N_c(\triv)-N_c(\sigma)))\beta_c(p,p) = \beta_c(E(c)p,p) = \sum_j\beta_c(y_jp,y_jp) > 0,
\end{equation}
which yields the claim, since any polynomial can be uniquely decomposed in terms of its degree and isotypic component.  For the other implication, assume $\beta_c$ is positive definite and let $0\neq p$ be of degree $m(\sigma)$
and in the $\sigma$-isotypic component. Then, (\ref{e:posdefeq}) implies that 
$(m(\sigma) - (N_c(\triv) - N_c(\sigma))>0$. This finishes the proof.
\end{proof}

\subsection{The Clifford algebra}\label{s:cliff}
Set $V = \fh\oplus\fh^*$ and extend the natural bilinear pairing $\fh^*\times\fh\to\bbc$ to a symmetric bilinear pairing $B:V\times V\to \bbc$, by declaring $B( y,y') = 0 = B( x,x')$, for all $y,y'\in\fh$ and $x,x'\in\fh^*$. Let $\gls{Cliff} = \Cc(V,B)$ be the 
Clifford algebra over $\bbc$ associated to $V$ and the bilinear pairing indicated, with Clifford relation
\begin{equation}\label{e:Cliff-relation}
\{v,v'\} = vv' + v'v = B(v,v'),
\end{equation}
for all $v,v'$ in $V$. Since $V = \fh\oplus\fh^*$ is even dimensional, $\Cc$ has $\gls{Spin} = \wedge\fh^*$ as the irreducible spin module with Clifford action $\sigma:\Cc\to \End_\bbc(\bbs)$ given by
\begin{equation}\label{e:S-Cliff}
\sigma(x)(\omega) = \mu_x(\omega) = x\wedge \omega,\qquad \sigma(y)(\omega) = \partial_y(\omega),
\end{equation}
for all $y\in\fh,x\in\fh^*$ and $\omega\in\bbs$. Here, $\partial_y$ is the contraction characterised by 
$\partial_y(x) = \lpi x, y\rpi,$ and extended to $\wedge \fh^*$ as an odd derivation. Note that:

\begin{proposition}\label{p:Cliffstd}
Fix an orthonormal basis $\{y_j\}\subseteq E$ of $\fh = E_\bbc$ and a dual basis $\{x_j\}\subseteq \fh^*$. Then, the
element $\sum_j x_jy_j\in\Cc$ acts as multiplication by $l\in\bbz_{\geq 0}$ on $\bbs^l = \wedge^l\fh^*$.
\end{proposition}

\begin{proof}
Follows from $\sum_j \sigma(x_j)\sigma(y_j)(x) = x$, for all $x\in \fh^*$.
\end{proof}

Similarly to what was defined for $\cha$, we extend the anti-linear isomorphism $\gls{Star}:\fh\to\fh^*$ to an anti-linear anti-involution on $\Cc=\wedge V$. The Hermitian product $(\cdot,\cdot)$ on $\fh^* = E^*_\bbc$ extends to a positive-definite Hermitian structure $( \cdot,\cdot )_\bbs$ on $\bbs$, which, on decomposable vectors of degree $l$, takes the form
$(x'_1\wedge\cdots\wedge x'_l, x''_1\wedge\cdots\wedge x''_l)_\bbs = \det((x'_j,x''_k)),$
with $x'_1,\ldots x'_l,x''_1,\ldots,x''_l\in\fh^*$. One checks that $(\cdot,\cdot)_\bbs$ satisfies
\[(\sigma(y)(\omega),\omega')_\bbs = (\omega,\sigma(y^*)\omega')_\bbs,\]
for all $y\in\fh$ and $\omega,\omega'\in\bbs$ and thus $(\bbs,(\cdot,\cdot)_\bbs)$ is a $*$-unitary $\Cc$-module.

\subsection{The $\fs\fp\fo(2|2,\bbc)$ Lie superalgebra}\label{s:spo}
We will follow the conventions from \cite{CW}. The Lie superalgebra $\fs\fp\fo(2|2,\bbc)$ is explicitly realised as the subspace of $\textup{Mat}_4(\bbc)$ given by
\[
\fs\fp\fo(2|2,\bbc)=\left\{
\begin{pmatrix}
	a&b&y&u\\c&-a&-v&-x\\x&u&d&0\\v&y&0&-d 
\end{pmatrix}
; a,b,c,d,x,y,u,v\in\bbc
\right\}\subseteq\textup{Mat}_4(\bbc).\]
We will fix the following basis for $\fg:=\fs\fp\fo(2|2,\bbc)$, where $E_{ij}\in \textup{Mat}_4(\bbc)$ stands for the elementary matrix with $1$ in the
 entry $(i,j)$ and zero everywhere else:
\begin{equation}\label{e:spobasis}
\begin{array}{rclcrclcrclcrcl}
h     &:=& E_{11}-E_{22},&& e_1^+ &:=& E_{12},&&
e_2^+ &:=& E_{31}-E_{24},       && e_3^+ &:=& E_{14}+E_{32},\\
z     &:=& E_{33}-E_{44},&& e_1^- &:=& E_{21},&&
e_2^- &:=& E_{13}+E_{42},&& e_3^- &:=& E_{41}-E_{23}.
\end{array}
\end{equation}
The even part $\fg_{\bar 0}$ is spanned by $h,z,e_1^\pm$ and is isomorphic to $\fs\fp(2)\oplus\fo(2)$. They satisfy the
relations
\begin{equation}\label{e:gl2}
[h,e_1^\pm]=\pm 2e_1^\pm,\qquad[e_1^+,e_1^-]=h,\qquad [h,z]=0=[e_1^\pm,z].
\end{equation}
The odd part $\fg_{\bar 1}$ is spanned by $e_2^\pm, e_3^\pm$ and the adjoint action of $\fg_{\bar 0}$ on $\fg_{\bar 1}$
is expressed by the relations
\begin{equation}\label{e:evenaction}
\begin{array}{rclcrclcrclcrcl}
\;\! [h,e_2^\pm] &=&\mp  e_2^\pm,&&[z,e_2^\pm] &=&\pm e_2^\pm,&&
\;\! [h,e_3^\pm] &=&\pm  e_3^\pm,&&[z ,e_3^\pm] &=&\pm e_3^\pm,\\
\;\! [e_1^\pm,e_2^\mp] &=& 0,&&[e_1^\pm,e_2^\pm] &=& -  e_3^\pm ,&&
\;\! [e_1^\pm,e_3^\pm] &=& 0,&& [e_1^\pm,e_3^\mp] &=& -  e_2^\mp.
\end{array}
\end{equation}
Moreover, the odd elements satisfy the anti-commutation relations
\begin{equation}\label{e:anti-commutation}
\begin{array}{rclcrclcrcl}
\{e_2^+,e_2^-\}&=& h+z,&&\{e_3^+,e_3^-\}&=& h-z,&&
\{e_2^\mp,e_3^\pm\}&=&\pm2e_1^\pm ,\\
\{e_2^\pm,e_2^\pm\}&=& 0,&&\{e_3^\pm,e_3^\pm\}&=& 0,&&
\{e_2^\pm,e_3^\pm\}&=& 0.              
\end{array}
\end{equation}
The equations above provide a complete set of relations for $\fs\fp\fo(2|2,\bbc)$. They
can be better understood in terms of a root space decomposition of $\fg$. Let $\ft=\bbc h\oplus \bbc z$ be a Cartan
subalgebra, realised as the subspace of diagonal matrices, and let $\{\delta,\epsilon\}\subseteq \ft^*$ be the basis dual to $\{h,z\}$. The root system of $\fg$ with respect to $\ft$ is $\Phi = \{\pm 2\delta\}\cup\{\pm\delta\pm\epsilon\}$ and
we will declare the roots $\alpha_1:=2\delta,\alpha_2:=\epsilon-\delta$ and $\alpha_3:=\epsilon+\delta$ to be the positive ones. The root space decomposition is $\fg=\ft\oplus\big(\oplus_\alpha\fg_\alpha\big)$ with $\fg_{\pm\alpha_i}$ spanned
by $e_i^\pm$. The (super) commutation relations $[\fg_{\alpha},\fg_{\beta}]\subseteq\fg_{\alpha+\beta}$ are contained
in the equations (\ref{e:gl2}) to (\ref{e:anti-commutation}). With respect to the linear basis fixed in (\ref{e:spobasis}), we let 
\begin{equation}\label{e:nilpot}
\gls{UNilp}:=\bbc e_1^-\oplus\bbc e_3^-\qquad\textup{ and }\qquad
\gls{NNilp}:= \fu^-\oplus\bbc e_2^-. 
\end{equation}
Similarly, we define $\fn^+$ and $\fu^+$.

\section{Lowest weight modules for $\mathfrak{spo}(2|2)$}\label{s:LWTheory}

\subsection{Standard and lowest weight modules}
Let $\fa = \fa_{\bar 0}\oplus \fa_{\bar 1}$ be a finite-dimensional super Lie algebra (over $\bbc$). We shall denote by $\Cu(\fa)$
its universal enveloping algebra. If $\{e_1,\ldots,e_t\}\subseteq \fa_{\bar 0}$ and $\{o_1,\ldots,o_s\}\subseteq \fa_{\bar 1}$
are linear bases, with $s,t\in\bbz_{\geq 0}$, then $\Cu(\fa)$ satisfies the PBW Theorem (see \cite[Theorem 1.36]{CW}):
\begin{theorem}\label{t:PBW}
The set $\{o_1^{q_1}\cdots o_s^{q_s}e_1^{p_1}\cdots e_t^{p_t}\mid q_1,\ldots,q_s\in\{0,1\},
 p_1,\ldots,p_t\in \bbz_{\geq 0}\}$ is a basis for $\Cu(\fa)$.
\end{theorem}
The algebra $\fg$ has a triangular decomposition
$\fg = \fn^+\oplus\ft\oplus\fn^-$,
where $\ft$ is the Cartan subalgebra of $\fg$. We let $\fb^- = \ft\oplus\fn^-$, so that 
$\Cu(\fg) = \Cu(\fn^+)\otimes\Cu(\fb^-)$. For each 
$\lambda=a\delta+b\epsilon \in \ft^*$, we let $\bbc_\lambda$ denote the simple $\ft$-module at
$\lambda$. The {\it standard module} at $\lambda$ is defined as
\[\bbv(\lambda) := \Cu(\fg)\otimes_{\Cu(\fb^-)}\bbc_\lambda.\]
Let $v_\lambda = 1\otimes 1\in \bbv(\lambda)$. From the PBW Theorem \ref{t:PBW}, it follows that the set
\begin{equation}\label{e:Vermabasis}
 \{(e_2^+)^{q_2}(e_3^+)^{q_3}(e_1^+)^{p}v_\lambda\mid q_2,q_3\in\{0,1\},p\in \bbz_{\geq 0}\}
 \end{equation}
 is a linear basis for $\bbv(\lambda)$. Throughout, by a {\it lowest weight module} for $\fg = \fs\fp\fo(2|2,\bbc)$, we shall mean a
 $\fg$-module generated by a single nonzero vector that is annihilated by $\fn^-$. Such modules are, necessarily, a quotient 
 of $\bbv(\lambda)$, for some $\lambda\in \ft^*$.

\subsection{Modules for the Grassmann algebra}\label{s:grassmann}
In this section, we shall follow \cite[Chapter 18, \textsection 2]{Ma}, although our discussion is slightly different since we consider a bigraded theory instead of just graded. Let $\Gamma=\wedge[e_2,e_3]$ be the Grassmann 
(or exterior) algebra in two generators. They satisfy the relations
\[e_2^2=0=e_3^2,\qquad e_2e_3 = - e_3e_2.\]
We shall view $\Gamma$ as a bigraded algebra with 
$\Gamma = \Gamma_{(0,0)}\oplus\Gamma_{(-1,1)}\oplus\Gamma_{(1,1)}\oplus\Gamma_{(0,2)}$ and
\[\Gamma_{(0,0)} = \bbc,\quad\Gamma_{(-1,1)}=\bbc e_2,\quad \Gamma_{(1,1)}=\bbc e_3,\quad 
\Gamma_{(0,2)}=\bbc e_2e_3 = \bbc e_3e_2.\]
The indecomposable modules of $\Gamma$ are well-understood. Up to isomorphism and degree shifts, 
an indecomposable $\Gamma$-module is either the free module on one 
generator or a {\it lightning-flash module}, which we recall the definition next:
\begin{definition}
Let $J\subseteq \bbz$ and $\delta_i\in\{0,1\}$ for $i=2,3$. Let also $\{x_j\mid j\in J\}$ be a set of
generators with $x_0$ in degree $(0,0)$, $x_j$ in degree $(2j,0)$ and put 
$F = \oplus_{j\in J}\Gamma x_j$. Define: 
\begin{enumerate}
\item $L(k,\delta_2,\delta_3)$, with $J = \{0,1,\ldots,k\}$, as the quotient of $F$ modulo 
the relations \[e_3x_j = e_2x_{j+1},\quad (1-\delta_2)e_2x_0=0,\quad (1-\delta_3)e_3x_k=0.\]
\item $\gls{Lightning}$, with $J = \{0,1,\ldots\}$, as the quotient of $F$ modulo 
the relations \[e_3x_j = e_2x_{j+1},\quad (1-\delta_2)e_2x_0=0.\]
\item $L(-\infty,\delta_3)$, with $J = \{\ldots,-1,0\}$, as the quotient of $F$ modulo 
the relations \[e_3x_j = e_2x_{j+1},\quad (1-\delta_3)e_3x_0=0.\]
\item $L(\infty)$, with $J = \{\ldots,-1,0,1,\ldots\}$, as the quotient of $F$ modulo 
the relation \[e_3x_j = e_2x_{j+1}.\]
\end{enumerate}
\end{definition}

\begin{theorem}[\cite{Ma}, Theorem 5, Chapter 18, \textsection 2]
Every $\Gamma$-module has a decomposition, unique up to isomorphism, as the coproduct of a free module
and the coproduct of lightning flash modules. 
\end{theorem}

\begin{definition}
Let $M$ be a $\Gamma$-module.  For $j\in\{2,3\}$ define the vector space $H(M,e_j) := \ker(e_j)/\im(e_j)$. 
We shall call them the {\it $e_j$-homology of $M$}. If $M$ is bigraded, the homologies inherit the grading.
\end{definition}

\subsection{Lowest weight modules and the Grassmann algebra}
From (\ref{e:anti-commutation}), the elements $e_2^+,e_3^+$ generate a copy of $\Gamma$ inside $\Cu(\fn^+)$. 
The purpose of this section is to classify, in terms of indecomposable modules for the Grassmann algebra 
$\Gamma \cong \wedge[e_2^+,e_3^+]$, the types of lowest weight modules that will be of interest to us. 
Throughout, we shall view  $\Cu(\fn^+)$ 
as a bigraded algebra by declaring $e_1^+$ to be of degree $(2,0)$ and $e_2^+,e_3^+$ to be of degree $(-1,1)$ and
$(1,1)$, respectively.

\begin{proposition}\label{p:VermaG}
Let $\lambda\in \ft^*$. The standard module $\bbv = \bbv(\lambda)$ is a free $\Gamma$-module on the generators 
$\{(e_1^+)^pv_\lambda\mid p\in\bbz_{\geq 0}\}$. If we declare $v_\lambda$ to
be of degree $(0,0)$, then 
$\bbv = \oplus_{p\in\bbz_{\geq 0}}(\bbv_{(2p,0)}\oplus\bbv_{(2p-1,1)}\oplus\bbv_{(2p,2)})$ with
\[\dim\bbv_{(2p,0)}=1=\dim\bbv_{(2p,2)},\quad\dim\bbv_{(-1,1)}=1,\qquad\dim\bbv_{(2p+1,1)}=2, \]
for all $p\in\bbz_{\geq 0}$.
\end{proposition}

\begin{proof}
From the description (\ref{e:Vermabasis}) of the linear basis of $\bbv$, each $v_{2p}:=(e_1^+)^pv_\lambda$ generates a $\Gamma$-module isomorphic to $\Gamma$. Moreover, $(e_2^+)^{q_2}(e_3^+)^{q_3}(e_1^+)^{p}v_\lambda$ is of degree 
$(-q_2+q_3+2p,q_2+q_3)$, with $p\in\bbz_{\geq 0}$ and $q_2,q_3\in \{0,1\}$. Counting the occurrences of a bidegree $(m,n)$ we 
get
\begin{equation}\label{e:vermadegrees}
\begin{array}{rclcrcl}
\bbv_{(2p,0)} &=& \bbc((e_1^+)^pv_0), && \bbv_{(2p,2)} &=& \bbc (e_2^+e_3^+(e_1^+)^pv_0),\\
\bbv_{(-1,1)} &=& \bbc (e_2^+v_0), && \bbv_{(2p+1,1)} &=& \bbc (e_3^+(e_1^+)^{p}v_0)\oplus(e_2^+(e_1^+)^{p+1}v_0), \\
\end{array}
\end{equation}
which yields the claim on the dimensions. 
\end{proof}

\begin{remark}
The $(p,q)$-degree space of $\bbv(\lambda)$ corresponds to the $(\lambda + p\delta + q\epsilon)$-weight space.
\end{remark}

\begin{corollary}
Let $L$ be a lowest weight module for $\fg$ generated by $v_0$. If we declare $v_0$ to be of degree 
$(0,0)$, then, as a graded $\Gamma$-module, we have $L = \oplus_{p\in\bbz_{\geq 0}}(L_{(2p,0)}\oplus 
L_{(2p-1,1)}\oplus L_{(2p,2)})$ and 
\[\dim L_{(2p,0)}\leq 1,\quad \dim L_{(2p,2)}\leq 1,\quad\dim L_{(-1,1)}\leq 1,
\qquad\dim L_{(2p+1,1)}\leq2, \]
for all $p\in\{0,1,\ldots\}$.
\end{corollary}

Let $\fp^-\subseteq \fg_{\bar 0}$ denote the subalgebra of the even part spanned by $h,z,e_1^-$. If 
$\mu\in\ft^*=\bbc\delta\oplus\bbc\epsilon$, we shall denote by 
$\bbu(\mu) = \Cu(\fg_{\bar 0})\otimes_{\Cu(\fp^-)}\bbc_\mu$ the corresponding standard module of $\fg_{\bar 0}$ at $\mu$.

\begin{proposition}\label{p:VermaDec}
Let $\lambda\in\ft^*$. The restriction of $\bbv(\lambda)$ to its even part $\fg_{\bar 0} \cong \fg\fl(2)$ decomposes as
\[\bbv(\lambda) = \bbu(\lambda)\oplus \bbu(\lambda+\epsilon-\delta)\oplus \bbu(\lambda+2\epsilon)\oplus 
\bbu(\lambda+\epsilon+\delta).\]
\end{proposition}

\begin{proof}
Let $v_1 := v_\lambda$, $v_2 := e_2^+v_\lambda$, $v_3 := e_2^+e_3^+v_\lambda$ and 
$v_4 := e_2^+e_1^+v_\lambda - \lambda(h)e_3^+v_\lambda$. It is clear that $\{v_1,v_2,v_3,v_4\}$ is a linearly 
independent set, from the description of the PBW basis in (\ref{e:Vermabasis}). Moreover, these vectors are in the weight
spaces of $\bbv(\lambda)$ of weight $\lambda,\lambda+\alpha_2,\lambda+(\alpha_2+\alpha_3)$ and $\lambda+\alpha_3$,
respectively.

Using Proposition \ref{p:VermaG}, 
the claim will follow if we show that these vectors are killed by $e_1^-$. Since $[e_1^-,e_2^+]=0$, we have 
$e_1^-v_1=0=e_1^-v_2$. Further, 
$e_1^-v_3 = e_2^+[e_1^-,e_3^+]v_\lambda = -(e_2^+)^2v_\lambda = 0$.  Finally, note that
\[
e_1^-v_4 =  e_2^+[e_1^-,e_1^+]v_\lambda -\lambda(h)[e_1^-,e_3^+]v_\lambda = -\lambda(h)e_2^+v_\lambda + \lambda(h)e_2^+v_\lambda = 0,
\]
as required.
\end{proof}

\begin{remark}
In \cite[Proposition 2]{BDSES} a similar discussion for the decomposition of the Verma module for
the super Lie algebra $\fs\fl(1|2)\cong\fs\fp\fo(2|2)$ was made.
\end{remark}

\begin{proposition}\label{p:LWmodules}
Let $L$ be a lowest weight module for $\fs\fp\fo(2|2,\bbc)$ with lowest weight $\lambda\in\ft^*$ and generated by
a lowest weight vector $v_\lambda$, in degree $(0,0)$. Suppose that:
\begin{enumerate}
\item[(a)] $e_1^+$ acts injectively,
\item[(b)] the homology $H(L,e_2^+)$ is concentrated in degree $(0,0)$,
\item[(c)] the homology $H(L,e_3^+)$ is concentrated in degree $(-1,1)$.
\end{enumerate}
Then, as a $\Gamma$-module, $L$ is generated by $J = \{(e_1^+)^pv_\lambda\mid p\in\bbz_{\geq 0}\}$. Moreover, if both homologies are trivial, then $L$ is free.  
If $H(L,e_2^+)\neq 0$, then $L\cong L(+\infty,0)$. If $H(L,e_3^+)\neq 0$, then $L\cong L(+\infty,1)$.
\end{proposition}

\begin{proof}
The injectivity of $e_1^+$ implies that the set $J$ in the statement is indeed a set of generators of $L$ as a $\Gamma$-module. 
Next, the conditions on the degree in which the $e_3^+$-homology appears imply that $e_3^+(e_1^+)^pv_0\neq 0$  for all $p\in\bbz_{\geq 0}$. The condition on the degree of the $e_2^+$-homology imply that, if nonzero, then $H(L,e_2^+)\cong L_{(0,0)}$, so $e_2^+v_0 = 0$ and $H(L,e_3^+) = 0$. In this case, using  (\ref{e:vermadegrees}) and the easy identity $[e_2^+,(e_1^+)^{p+1}] =(p+1)e_3^+(e_1^+)^p$, we have $\dim L_{(2p,0)} = 1 = \dim L_{(2p+1,1)}$ and 
$\dim L_{(2p,2)} = 0 = \dim L_{(-1,1)}$,  for all $p\in\bbz_{\geq 0}$, so that $L\cong L(+\infty,0)$. If $H(L,e_3^+)\neq 0$, then
$e_3^+e_2^+v_0 = e_2^+e_3^+v_0 = 0$. Using again $[e_2^+,(e_1^+)^{p}] =pe_3^+(e_1^+)^{p-1}$,
it follows that $e_2^+e_3^+(e_1^+)^pv_0 = 0$ from which $\dim L_{(2p,2)} = 0$,   for all $p\in\bbz_{\geq 0}$ and also $e_3^+(e_1^+)^pv_0\in e_2^+(\bbc (e_1^+)^{p+1}v_0)$, by exactness of $e_2^+$. Hence, $\dim L_{(2p,0)} = 1 = \dim L_{(2p-1,1)}$ from which $L\cong L(+\infty,1)$. If both are zero, then $L$ is free.
\end{proof}

\begin{definition}\label{d:LWmodules}
Let $\lambda\in\ft^*$. We shall denote by $L(\lambda)$ any irreducible lowest weight representation of $\fs\fp\fo(2|2,\bbc)$
of lowest weight $\lambda$. If $L(\lambda)$ is such a module, in view of Proposition \ref{p:LWmodules}, we shall write
$L(\lambda) = L_0(\lambda)$, $L(\lambda) = L_1(\lambda)$ or $L(\lambda) = \bbv(\lambda)$ if $L(\lambda)$ is 
isomorphic to $L(+\infty,0)$, $L(+\infty,1)$ or if it is a free $\Gamma$-module, respectively.
\end{definition}

\section{The realisation of $\fs\fp\fo(2|2,\bbc)$ inside $\cha\otimes\Cc$}\label{s:realisation}

\subsection{The realisation}
In this subsection, we will describe a realisation of the Lie super algebra $\fg = \fs\fp\fo(2|2,\bbc)$ inside the algebra 
$\cha\otimes\Cc$. We start by describing an embedding of $\bbc W$ into $\Cc$, which is particular to the situation
of a Clifford algebra in even dimension. For each $\alpha\in R_+$  let $\alpha\in\fh^*$ and $\alpha^\vee\in\fh$.
Consider the elements $\gls{Tau}\in \Cc$,  which were first introduced in \cite[Section 4.5]{Ci}, and are given by
\begin{equation}\label{e:tau}
\tau_\alpha := 1 - \alpha\alpha^\vee = \alpha^\vee\alpha-1\in \Cc.
\end{equation}
\begin{proposition}\label{p:WEmbedding}
The elements $\{\tau_\alpha\mid \alpha\in R_+\}$ satisfy 
the relations $(\tau_\alpha\tau_\beta)^{m(\alpha,\beta)}=1$,
where $m(\alpha,\beta)$ is the order of the product $s_\alpha s_\beta$ in $W$. The assignment $s_\alpha\mapsto \tau_\alpha$ extends to an injective homomorphism $W\to \Cc^\times$, the group of units of $\Cc$.
\end{proposition}
\begin{proof}
This is discussed in \cite[Section 4.5]{Ci}, but with a different normalisation for the Clifford relation. For convenience, we
spell out the direct computations here. We start by claiming that each $\tau_\alpha$ acts by $s_\alpha$ on the spin module.
Indeed, consider a basis $\{x_1,x_2,\ldots,x_r\}$ of $\fh^*$ such that $x_1 = \alpha$ and the remaining elements for $j>1$
satisfy $\lpi x_j,\alpha^\vee \rpi = 0$. Then, it is straightforward to check, using (\ref{e:S-Cliff}), that for all $1\leq l \leq r$ we have
\[
\sigma(\tau_\alpha)(x_{i_1}\wedge\cdots\wedge x_{i_l}) = s_\alpha(x_{i_1})\wedge\cdots\wedge s_\alpha(x_{i_l}),
\]
where $x_{i_1}\wedge\cdots\wedge x_{i_l}$ is a basis element of $\bbs^l = \wedge^l\fh^*$: if $i_1>1$, then all $x_{i_k}$ are annihilated by $\alpha^\vee$
so that $\tau_\alpha = 1 - \alpha\alpha^\vee$ acts as the identity whereas if $x_{i_1}=\alpha$, then $\tau_\alpha$ acts as minus the identity on that basis vector. Since the action of $\Cc$ on
$\bbs$ is faithful, any relation $(s_\alpha s_\beta)^m = 1$ for a positive integer $m$ and positive roots $\alpha,\beta$ imply
that $(\tau_\alpha\tau_\beta)^m = 1$ in $\Cc$, and we are done.

Because that $\tau_\alpha$'s satisfy the defining relations of $W$, the assignment $s_\alpha\mapsto \tau_\alpha$ extends to a group homomorphism $W\to \Cc^\times$, $w\mapsto \tau_w$. To see that this is injective, we use the action of $\tau_w$ on $\mathbb S$, in fact only on the $\mathfrak h^*$ part: $\sigma(\tau_w)(x)=w(x)$. If $\tau_w=1$, it follows that $w(x)=x$ for all $x\in \mathfrak h^*$, and since $\mathfrak h^*$ is a faithful $W$-representation, $w=1$.
\end{proof}
Extend the homomorphism $W\to \Cc^\times$ from Proposition \ref{p:WEmbedding} to an algebra homomorphism $\bbc W\to \Cc$. Denote by $\gls{Diag}:\bbc W\to\cha\otimes \Cc$ the diagonal embedding. Define 
\begin{equation}\label{e:omega}
\gls{Omega} := \sum_{\alpha>0} c_\alpha s_\alpha \otimes \tau_\alpha \in (\cha\otimes\Cc)^W.
\end{equation}
Now, fix bases $\gls{Basx}\subseteq E^*$ and $\gls{Basy}\subseteq E$. We recall that we have identified $E$ and $E^*$
by means or the Euclidean structure on $E$, so we can and will assume that these bases are orthonormal and correspond to each
other in the above-mentioned identification.  Put $h_i:=\{x_i,y_i\}\in\cha$ and $z_i:=[x_i,y_i] = (2x_iy_i-1)\in\Cc$. Define the following elements in $\cha\otimes\Cc$:
\begin{equation}\label{e:spodefs}
\begin{array}{rclcrcl}
H    &:=&\sum_i\tfrac{1}{2} h_i\otimes 1,     &\quad&Z  &:=&\sum_i 1\otimes \tfrac{1}{2}z_i + \Omega_c,\\
E_1^+&:=&-\sum_i \tfrac{1}{2}x_i^2\otimes 1,   &\quad&E_1^-&:=&\sum_i \tfrac{1}{2}y_i^2\otimes 1,\\
E_2^+&:=&\sum_i y_i\otimes x_i, &\quad&E_2^-&:=&\sum_i x_i\otimes y_i,\\
E_3^+&:=&-\sum_i x_i\otimes x_i, &\quad&E_3^-&:=&-\sum_i y_i\otimes y_i.
\end{array}
\end{equation}
Let $\gls{SPO}$ be the linear span of the elements $H,Z,E_1^\pm,E_2^\pm,E_3^\pm$ defined in (\ref{e:spodefs}). The rest of this subsection is devoted to proving the following result:

\begin{theorem}\label{t:spo}
The subspace $\fg$ of $\cha\otimes \Cc$ is a Lie superalgebra isomorphic to $\fs\fp\fo(2|2,\bbc)$.
\end{theorem}

\begin{remark}
Setting $c=0$ yields a realisation of $\fs\fp\fo(2|2,\bbc)$ inside the Weyl-Clifford algebra conjugate to the general description
of the algebras $\fs\fp\fo(2m|2n,\bbc)$ described in \cite[Section 5.3.4]{CW}.
\end{remark}

\begin{remark}
In comparison with \cite{DO} and \cite{DJO}, the image of $E_2^+$ and $E_2^-$ inside $\End(M_c(\triv)\otimes\bbs)$, correspond to the deformed exterior derivative $d(c)$ and the boundary operator $\partial$ acting on polynomial differential forms on $\fh$.
\end{remark}

\begin{lemma}\label{l:Winvariance}
The elements defined in (\ref{e:spodefs}) are contained in $(\cha\otimes\Cc)^W$.
\end{lemma}

\begin{proof}
The claim follows because the bases $\{x_i\}\subseteq E^*$ and $\{y_i\}\subseteq E$ are orthonormal and the element $\Omega_c$ is Weyl-group invariant.
\end{proof}

\begin{proposition}
The elements $H,Z,E_1^\pm$ span a copy of $\fg\fl(2)$ inside $\cha\otimes \Cc$. 
\end{proposition}

\begin{proof}
The fact that  ${H, E_1^\pm}$ form an $\fs\fl(2)$-triple was discussed in \cite[Theorem 3.3]{He}. Certainly, we have that this $\fs\fl(2)$-triple commute with $Z_0 = Z-\Omega_c= \sum_i1\otimes \tfrac{1}{2}z_i\in \Cc$. The claim will follow if they also 
commute with $\Omega_c$. But note that if $p\in\cha^W$, then 
$\Omega_c (p\otimes 1) = \sum_{\alpha >0} c_\alpha s_\alpha(p) s_\alpha\otimes \tau_\alpha = (p\otimes 1)\Omega_c$. Since this is the case for $H,E_1^\pm$, we are done.
\end{proof}

\begin{proposition}\label{p:gradings}
Let $Z_0 = Z-\Omega_c = \sum_i1\otimes \tfrac{1}{2}z_i\in \Cc$ and $x\in \fh^*, y\in \fh$, Let also $p\in\cha$ and $\omega\in\Cc$. The following relations hold in $\cha\otimes \Cc$:
\[
\begin{array}{rclcrcl}
\;\![H,x\otimes \omega] &=&x\otimes \omega,     &\quad&[H,y\otimes \omega] &=&-y\otimes \omega,\\
\;\![Z_0,p\otimes x] &=& p\otimes x,   &\quad&[Z_0,1\otimes y] &=&-p\otimes y.
\end{array}
\]
\end{proposition}

\begin{proof}
In the case of a real reflection group, it is known that the element $H$ is the grading element of $\cha$ 
(see \cite[Section 3.1]{GGOR} or \cite[Proposition 3.18]{EM}), which yields the equations in the first row. As for the second
row, note that in $\Cc$, we have
\[
[z_i,x] = (x_iy_i - y_ix_i)x - x(x_iy_i - y_ix_i) = 2(x_iy_ix + x_ixy_i) = 2\lpi x,y_i \rpi x_i,
\] 
from which $[Z_0,p\otimes x] = p\otimes [Z_0,x] = p\otimes x$. Similarly, $[z_i,y] = -y$ and $[Z_0,p\otimes y] = -p\otimes y$.
\end{proof}

\begin{remark}\label{r:clifftotaldeg}
We can view $Z_0$ as a $\bbz$-grading element on $\Cc$, since $[Z_0,x]=x$ and $[Z_0,y]=-y$ for $x\in\fh^*$ and $y\in\fh$.
We have a decomposition as graded vector space $\Cc = \oplus_{n=-r}^r\Cc_n$, with $\gls{Cliffn}$ the degree-$n$ space.
\end{remark}

\begin{corollary}
The relations $[H,E_2^{\pm}] = \mp E_2^\pm, [H,E_3^{\pm}] = \pm E_2^\pm, [Z,E_2^{\pm}] = \pm E_2^\pm$ and $[Z,E_3^{\pm}] = \pm E_3^\pm$ hold in $\cha\otimes\Cc$.
\end{corollary}

\begin{proof}
The relations involving $H$ are immediate from Proposition \ref{p:gradings}. By the Weyl-group invariance of Lemma \ref{l:Winvariance}, we have $[Z,E_j^\pm] = [Z_0,E_j^\pm]$, for $j\in\{2,3\}$ and the claim now follows from Proposition \ref{p:gradings}.
\end{proof}

\begin{lemma}\label{l:Hamiltonian}
Let $p = \sum_j x_j^2$ and $q = \sum_k y_j^2$ in $\cha$. Then $[-p,y_k] = 2x_k$ and $[q,x_k] = 2y_k$.
\end{lemma}

\begin{proof}
Note that if $p_0\in\bbc[\fh],q_0\in\bbc[\fh^*]$ and $x\in\fh^*,y\in\fh$, then
\[
[y,p_0] = \partial_y(p_0) - \sum_{\alpha >0}c_\alpha\lpi \alpha,y \rpi\Delta_\alpha(p_0)\qquad\textup{and}\qquad
[q_0,x] = \partial_x(q_0) - \sum_{\alpha >0}c_\alpha\lpi x,\alpha^\vee \rpi\Delta_\alpha(q_0).
\]
This is easily proven by induction on the degree of the polynomials involved and the details can be found in, for example, \cite[Propositions 2.5 and 2.6]{CDM}. Since $p$ and $q$ in the statement are $W$-invariant, they are killed by the operators $\Delta_\alpha$ which yields 
the claim.
\end{proof}

\begin{proposition}
The relations $[E_1^\pm,E_2^\pm] = -E_3^\pm, [E_1^\pm,E_3^\mp] = -E_2^\mp$ and $[E_1^\pm,E_2^\mp] = 0 = 
[E_1^\pm,E_3^\pm]$ hold in $\cha\otimes\Cc$.
\end{proposition}

\begin{proof}
Using that $[-\sum_jx_j^2,y_k]=2x_k$, from Lemma \ref{l:Hamiltonian}, we compute
\[
[E_1^+,E_2^+] = \sum_jx_j\otimes x_j = -E_3^+\qquad\textup{and}\qquad
[E_1^+,E_3^-] = -\sum_jx_j\otimes y_j = -E_2^-.
\]
Similarly, now using $[\sum_jy_j^2,x_k]=2y_k$, we get $[E_1^-,E_2^-] = -E_3^-$ and $[E_1^-,E_3^+] = -E_2^+$. The
other equations follows from $[x_j^2,x_k]=0=[y_j^2,y_k]$, for all $j,k$.
\end{proof}
So far, we have shown that the relations in (\ref{e:gl2}) and (\ref{e:evenaction}) are satisfied.To finish the proof of Theorem \ref{t:spo}, we need to establish the anti-commutation relations in (\ref{e:anti-commutation}). The relations $\{E_2^\pm,E_2^\pm\}=0=\{E_3^\pm,E_3^\pm\}$ are easy to verify, since in these cases, the right-hand side of the tensor
anti-commutes while the left-hand side commutes. Further, using the Clifford relation $\{x_i,y_j\}=\delta_{ij}$ we get
\[\{E_2^+,E_3^-\} = -\sum_{ij}y_iy_j\otimes\{x_i,y_j\}=-2E_1^-\qquad\textup{and}\qquad
\{E_2^-,E_3^+\} = -\sum_{ij}x_ix_j\otimes\{y_i,x_j\}=2E_1^+.
\]
For the last anti-commutation equations $\{E_2^+,E_2^-\} = H + Z$,  $\{E_3^+,E_3^-\} = H - Z$ and $\{E_2^\pm,E_3^\pm\} = 0$, we need the following:

\begin{lemma}\label{l:EucSymm}
With the choices of bases made, we have $[y_i,x_j]=[y_j,x_i]$, for all $i,j$.
\end{lemma}

\begin{proof}
The point here is that the bases of $\fh$ and $\fh^*$ were chosen to be orthonormal in $E$ and $E^*$ and they relate to each other under the identification of $E$ and $E^*$. Hence, we can write 
$\lpi x,y_i \rpi = (x,x_i)$, for all $i$ and $x\in E^*$. Also, we can view $\alpha^\vee = 2\alpha/(\alpha,\alpha)\in\fh^*$.
Then, 
\[
[y_i,x_j]-[y_j,x_i] = -\sum_{\alpha>0}\frac{2c_\alpha}{(\alpha,\alpha)}((\alpha,x_i)(x_j,\alpha) - ( \alpha,x_j)(x_i,\alpha))s_\alpha = 0,
\]
as required.
\end{proof}

As in the proof of the previous lemma, we may write $\lpi x,y_i \rpi = (x,x_i)$ for any $x\in E^*$.
Using the identities $\sum_i (\alpha,x_i )x_i =\alpha\in E^*\subseteq\fh^*$ and $\sum_j ( x_j,\alpha^\vee ) y_j =\alpha^\vee\in E\subseteq\fh$, we get
\begin{align*}
\{E_2^+,E_2^-\} &= \sum_{ij}\big(y_ix_j\otimes x_iy_j + x_jy_i\otimes y_jx_i + (x_jy_i\otimes x_iy_j - x_jy_i\otimes x_iy_j)\big)\\
&=\sum_{ij}[y_i,x_j]\otimes x_iy_j + x_jy_i\otimes \{y_j,x_i\}\\
&=\sum_{ij}(\delta_{ij} - \sum_{\alpha>0}c_\alpha ( \alpha,x_i )( x_j,\alpha^\vee ) s_\alpha)\otimes x_iy_j + \sum_jx_jy_j\otimes 1\\
&=\sum_{j}1\otimes x_jy_j - \sum_{\alpha>0}c_\alpha s_\alpha\otimes \alpha\alpha^\vee + \sum_jx_jy_j\otimes 1.
\end{align*}
Adding and subtracting $(\sum_{\alpha>0} c_\alpha s_\alpha\otimes 1 + \tfrac{r}{2}\otimes 1)$, where $r=\dim\fh$, yields 
\[
\{E_2^+,E_2^-\} = \sum_j1\otimes(x_jy_j-\tfrac{1}{2}) - \sum_{\alpha>0}c_\alpha s_\alpha\otimes(\alpha\alpha^\vee - 1) + \sum_{j}(x_jy_j + \tfrac{1}{2})\otimes 1 - \sum_{\alpha>0}c_\alpha s_\alpha\otimes 1 = Z_0+\Omega_c+H.
\]
Similarly, but now using Lemma \ref{l:EucSymm}
\begin{align*}
\{E_3^+,E_3^-\} &=\sum_{ij}[y_j,x_i]\otimes y_jx_i + x_iy_j\otimes \{x_i,y_j\}\\
&=\sum_{ij}[y_i,x_j]\otimes y_jx_i + x_iy_j\otimes \{x_i,y_j\}\\
&=\sum_{ij}(\delta_{ij} - \sum_{\alpha>0}c_\alpha (\alpha,x_i )( x_j,\alpha^\vee ) s_\alpha)\otimes y_jx_i + \sum_jx_jy_j\otimes 1\\
&=\sum_{j}1\otimes (y_jx_j-\tfrac{1}{2}) - \sum_{\alpha>0}c_\alpha s_\alpha\otimes (\alpha^\vee\alpha - 1) + \sum_j(x_jy_j +\tfrac{1}{2})\otimes 1 - \sum_{\alpha>0} c_\alpha s_\alpha)\otimes 1\\
&=-(Z_0+\Omega_c)+H.
\end{align*}
For the last relation $\{E_2^\pm,E_3^\pm\} = 0$, using Lemma \ref{l:EucSymm} again, we compute
\[\{E_2^+,E_3^+\}=\sum_{ij}y_ix_j\otimes x_ix_j = \sum_{i<j}([y_i,x_j] - [y_j,x_i])\otimes x_ix_j = 0.\]
Similarly, $\{E_2^-,E_3^-\} = 0$, and this finishes the proof of Theorem \ref{t:spo}.

\begin{remark}
In \cite{DBOVJ}, it is studied many realisations of the Lie super algebra $\fs\fp\fo(2|1)$ inside an algebra $\Ca\otimes\Cc$, where $\Cc$ is a suitable Clifford algebra and 
$\Ca$ is a general algebra possessing only the commutativity of the variables in $\fh$ and $\fh^*$ and the condition of Lemma \ref{l:Hamiltonian} (the rational Cherednik algebra is an example of such $\Ca$). They
also study the symmetry algebra (centraliser) of these realisations of $\fs\fp\fo(2|1)$ inside $\Ca\otimes\Cc$. It is natural to ask if similar considerations for $\fs\fp\fo(2|2)$ and their symmetry algebras can be made in that context.
\end{remark}

\subsection{A $\ast$-structure on $\fs\fp\fo(2|2,\bbc)$}
The $\gls{Star}$-structures on $\cha$ and on $\Cc$ described in Section \ref{s:defs} naturally induce a $\ast$-structure on the tensor product $\cha\otimes \Cc$. It is straightforward to check that, with respect to the choice of orthonormal bases $\{x_i\}\subseteq E^*$ and $\{y_j\}\subseteq E$ made, we have
\begin{equation}\label{e:star}
(y_j\otimes 1)^* = x_j\otimes 1 \qquad\textup{and}\qquad (1 \otimes y_j )^* = 1 \otimes x_j.
\end{equation}
We endow the space $\gls{Kmod} := M_c(\tau)\otimes\bbs$ with the Hermitian form $\gls{HermStruc} = \lpi\cdot|\cdot\rpi_{c,\tau}$ given by the product between the corresponding
forms, discussed in Section \ref{s:defs}.

\begin{proposition}
The $\ast$-structure on $\cha\otimes\Cc$ restricts to a $\ast$-structure on $\fg = \fs\fp\fo(2|2,\bbc)$, defined in (\ref{e:spodefs}).
\end{proposition}

\begin{proof}
Using (\ref{e:star}), it is immediate that $(E_1^\pm)^* = -E_1^\mp, (E_2^\pm)^*= E_2^\mp, (E_3^\pm)^*= E_3^\mp$ and $H^* = H$. It is also straightforward that $Z_0^* = Z_0$, since in $\Cc$, we have $[x_j,y_j]^* = [x_j,y_j]$, for all
$1\leq j \leq r$. Now, from
\[
(\alpha\alpha^\vee)^* = \sum_{jk}(\alpha,y_j)(\alpha^\vee,x_k)(x_jy_k)^* = \sum_{jk}\tfrac{2}{(\alpha,\alpha)}(\alpha,y_j)(\alpha,x_k)x_ky_j = \alpha\alpha^\vee,
\]
we obtain $\tau_\alpha^*=\tau_\alpha$ and hence $\Omega_c^* = \Omega_c$, which together with $Z_0^*=Z_0$ implies $Z^* = Z$.
It is straightforward to check that $\ast$ preserves the relations (\ref{e:gl2}) up to (\ref{e:anti-commutation})
\end{proof}

\begin{corollary}
For every irreducible $W$ representation $\tau$, the space $(K_c(\tau),\lpi\cdot|\cdot\rpi)$ is a $\ast$-Hermitian $\fs\fp\fo(2|2,\bbc)$-module. In the notation of Lemma \ref{l:unitarity}, if $c\in U(\tau)$, then $K_c(\tau)$ is a unitarisable 
$\fs\fp\fo(2|2,\bbc)$-module.
\end{corollary}

\section{The decomposition of $K_c(\tau)$}

In this section, we shall prove our main theorem. Fix, once and for all, $\tau\in\widehat{W}$ and $c\in \mathscr{C}$, so that
$(K_c(\tau),\lpi\cdot|\cdot\rpi)$ is a $\ast$-Hermitian $\fg =\fs\fp\fo(2|2,\bbc)$-module. For convenience, if necessary, we shall omit $c$ and $\tau$ from the notations.

\subsection{A coarse decomposition} 
For each pair of integers $(m,l)$ we let
$\gls{gKmod}:=\bbc[\fh]_m\otimes V(\tau)\otimes \bbs^l$, where $\bbs^l$ denotes the $l$-th exterior power of 
$\fh^*$. For $\sigma\in\widehat{W}$ we let $\gls{gKmodIso}$ denote the $\sigma$-isotypic component of $K^l_m$. We have 
\begin{equation}\label{e:weightDecomp}
K_c(\tau) = \bigoplus_{m,l\in\bbz_{\geq 0},\sigma\in\widehat{W}} K^l_m(\sigma).
\end{equation}
Note that (\ref{e:weightDecomp}) is in fact the weight space decomposition of $K$ with weights
 defined by the following (recall that $r = \dim \fh$):

\begin{proposition}
The generators $H,Z$ of the Cartan subalgebra of $\fg$ act on $\theta\in K^l_m(\sigma)$ via
$\gls{Lambda}\in\ft^*$ characterised by
\begin{equation}\label{e:weights}
H\theta = (m + \tfrac{r}{2} - N_c(\tau))\theta\qquad\textup{and}\qquad
Z\theta = (l - \tfrac{r}{2} + N_c(\sigma))\theta.
\end{equation}
\end{proposition}
\begin{proof}
It is a standard computation that 
$\sum_j x_jy_j + \tfrac{r}{2} - \sum_{\alpha>0}c_\alpha s_\alpha$ act by the scalar 
$m + \tfrac{r}{2} - N_c(\tau)$ on the degree $m$-part of $M_c(\tau)$. Moreover, since $\Omega_c$ is the image
of the central element $\sum_{\alpha >0}c_\alpha s_\alpha \in \bbc W$ under the diagonal embedding 
$\rho:\bbc W\to \cha\otimes \Cc$, it acts by $N_c(\sigma)$ on the $\sigma$-isotypic component. We are left to show that 
$\sum_j (x_jy_j - \tfrac{1}{2})$ acts on $\bbs^l$ by the scalar $l-\tfrac{r}{2}$. But this is indeed the case, 
from Proposition \ref{p:Cliffstd}.
\end{proof}

Note also that (\ref{e:weightDecomp}) is in fact an orthogonal decomposition:

\begin{proposition}\label{p:orthogonality}
We have $\lpi K_m^l(\sigma) | K_{m'}^{l'}(\sigma')\rpi = 0$, if $(m',l')\neq (m,l)$ or $\sigma'\neq \sigma$.
\end{proposition} 

\begin{proof}
The form $\lpi\cdot | \cdot\rpi$ is given by the product of the Hermitian structures on $M_c(\tau)$ and $\bbs$. The 
orthogonality with respect to $l$ is immediate from the definition of $(\cdot,\cdot)_{\bbs}$, while with respect to the polynomial degree $m$ and the isotypic component $\sigma$, it follows from \cite[Proposition 2.17]{DO}.
\end{proof}

\begin{definition}\label{d:harmonics}
Define
\begin{align*}
\Cm = \gls{Monog} &:= \{\theta\in K_c(\tau)\mid X\theta = 0, \textup{ for all } X\in\fn^-\subseteq \fg\}\\
\Ca = \gls{AHarm} &:= \{\theta\in K_c(\tau)\mid X\theta = 0, \textup{ for all } X\in\fu^-\subseteq \fg\}\\ 
\Ch = \gls{Harm}  &:= \{\theta\in K_c(\tau)\mid E_1^-\theta = 0 \}. 
\end{align*}
We shall call the nonzero elements of $\Cm_c(\tau)$ as the lowest weight vectors of $K_c(\tau)$.
\end{definition}

\begin{remark}
Since the generators of $\fn^-$ are homogeneous, we can decompose the spaces in Definition \ref{d:harmonics} with respect to bidegrees $(m,l)$. By $W$-invariance of these operators, we can decompose further with
respect to the isotypic components. We shall write $\Cm_m^l := \Cm\cap K_m^l$ and  
$\Cm_m^l(\sigma) := \Cm\cap K_m^l(\sigma)$. Similar for $\Ca$ and $\Ch$.
\end{remark}

From now onward, we shall assume that the parameter $c\in\mathscr{C}$ is chosen to be in the open neighbourhood $U(\tau)$
of $0\in\mathscr{C}$, in the notations of Lemma \ref{l:unitarity}, so that $K_c(\tau)$ is unitarisable. We also impose the
following generic condition on $c$:

\begin{assumption}\label{a:assumption}
Assume the parameter $c\in U(\tau)\subseteq \mathscr{C}$ is such that $\tfrac{r}{2} - N_c(\tau) \notin \bbz_{\leq 0}$ and also that $N_c(\tau)-N_c(\sigma)\notin\bbz_{\geq 0}$, for all $\sigma\in\widehat{W}\setminus\{\tau\}$.
\end{assumption} 

\begin{proposition}\label{p:rowdecomps}
The space $K^l_m$, of homogeneous elements of bidegree $(m,l)$  decomposes as
\[K^l_m = \bigoplus_{p=0}^{[m/2]} (E_1^+)^p(\Ch_{m-2p}^l),\]
where $[\cdot]$ denotes the integer part of a real number.
\end{proposition}

\begin{proof}
This result is the first part of \cite[Theorem 5.1]{BSO}. We present an argument, for convenience. 
Our argument is by induction on the polynomial degree. Suppose first that $m = 2k$ is even. For $k=0$ the result
is trivially true, so assume it for all $k'<k$, and  note that, because $\lpi\cdot|\cdot\rpi$ is positive definite, we have
\[K^l_{2k} = \Ch^l_{2k}\oplus (\Ch^l_{2k})^\perp.\]
The claim will follow by the inductive hypothesis, if we show that $(\Ch^l_{2k})^\perp = E_1^+(K^l_{2k-2})$. Certainly, 
$\lpi \Ch^l_{2k}|E_1^+(K^l_{2k-2})\rpi =0$, since $(E_1^+)^* = -E_1^-$. Now, let 
$U := \Ch^l_{2k}\oplus E_1^+(K^l_{2k-2})$. We shall argue that $U^\perp = 0$. Suppose there exists
$0\neq \theta\in U^\perp \subseteq K^l_{2k}$. If that was so, we would have $0 \neq  E_1^-\theta \in K^l_{2k-2}$, since
otherwise $\theta \in U\cap U^\perp$. But then, $E_1^+E_1^-\theta \in U$ and hence
\[ 
0 = \lpi \theta | E_1^+E_1^-\theta\rpi = -\lpi E_1^-\theta | E_1^-\theta\rpi<0, 
\]
a contradiction that shows $U^\perp = 0$. The argument for $m$ odd is entirely analogous.
\end{proof}

\begin{proposition}\label{p:e1harm}
The space $\Ch^l_m$ decomposes as an orthogonal sum
\[\Ch^l_m = \Cm_m^l \oplus E_2^+(\Cm_{m+1}^{l-1}) \oplus F_{m-1}(\Cm_{m-1}^{l-1}) \oplus E_2^+E_3^+(\Cm_m^{l-2}),\]
where the element $F_m\in\fn^+$, for $m\in\bbz_{\geq 0}$, is defined by
\[F_m:= E_2^+E_1^+ -\lambda_mE_3^+,\]
and $\lambda_m =  (m + \tfrac{r}{2} - N_c(\tau)) = \lambda_{m,l,\sigma}(H)$, which is nonzero, from the Assumption \ref{a:assumption}.
\end{proposition}

\begin{proof}
That each summand is in $\Ch_m^l$, follows from (the proof of) Proposition \ref{p:VermaDec}. Next, we show that the summands are orthogonal. If $\theta\in\Cm_m^l$ and $\psi\in K$, then, for $j\in \{2,3\}$, we have $\lpi \theta| E_j^+\psi\rpi = \lpi  E_j^-\theta|\psi\rpi = 0$, 
which implies that $\Cm_m^l$ is orthogonal to the other summands. Now, let $\theta\in\Cm_{m+1}^{l-1}$ and $\psi\in K$. Decompose $\theta = \sum_\sigma \theta_\sigma$, with respect to the isotypic components and
let $\lambda_\sigma$ be the scalar such that $(H+Z)\theta_\sigma = \lambda_\sigma\theta_\sigma$. Then
\[
\lpi E_2^+\theta | E_2^+E_1^+\psi\rpi = \lpi E_2^-E_2^+\theta | E_1^+\psi\rpi = -\lpi E_1^-(H+Z)\theta | \psi\rpi = -\sum_\sigma  \lambda_\sigma\lpi E_1^-\theta_\sigma | \psi\rpi =0,
\]
where we used that $E_2^-E_2^+\theta = (H+Z)\theta$ and $E_1^-\theta_\sigma =0$, because of $W$-invariance. Now using $E_3^-E_2^+\theta = -2E_1^-\theta = 0$, we have
$\lpi E_2^+\theta | E_3^+\psi\rpi  = 0$. Similarly, 
$\lpi E_2^+\theta | E_2^+E_3^+\psi\rpi = \sum_\sigma \lambda_\sigma\lpi E_3^-\theta_\sigma | \psi\rpi = 0,$ from which $ E_2^+(\Cm_{m+1}^{l-1})$ is orthogonal to the other summands.  Finally, if $\theta$ is in $\Cm^{l-1}_{m-1}$ and $\psi\in K$, we compute
\[
\lpi E_3^+\theta | E_2^+E_3^+\psi\rpi = \lpi E_2^-E_3^+\theta | E_3^+\psi\rpi = \lpi \{E_2^-,E_3^+\}\theta | E_3^+\psi\rpi = 0
\]
and 
\[
\lpi E_2^+E_1^+\theta | E_2^+E_3^+\psi\rpi = \lpi (H+Z)E_1^+\theta | E_3^+\psi\rpi - \lpi E_2^+E_2^-E_1^+\theta | E_3^+\psi\rpi = \sum_\sigma \lambda_\sigma\lpi E_3^-E_1^+\theta_\sigma | \psi\rpi = 0,
\]
since $[E_1^+,E_2^-]=0$ and $[E_3^-,E_1^+]=E_2^-$. This shows that the summands are indeed orthogonal.

Now, let $U\subseteq \Ch_m^l$ denote the sum in the right-hand side of the statement. We can write $\Ch_m^l = U\oplus U^\perp$. We claim that $U^\perp = 0$. Suppose there exists $0\neq \theta \in U^\perp\subseteq \Ch_m^l$. If it were the case that
$E_2^-\theta = 0$, then we would have $\theta\in\Cm_m^l$, so it must be that $E_2^-\theta\neq 0$. That said, if it were the case that $E_3^-\theta = 0$, then, $E_1^-E_2^-\theta = [E_1^-,E_2^-]\theta = 0$, from which we would conclude that 
$0\neq E_2^-\theta \in\Cm_{m+1}^{l-1}$. But,
\[0 < \lpi E_2^-\theta | E_2^-\theta\rpi = \lpi \theta | E_2^+(E_2^-\theta)\rpi = 0,\]
since $E_2^+(E_2^-\theta) \in E_2^+(\Cm_{m+1}^{l-1}) \subseteq U$ which is orthogonal to $U^\perp$. So it must be the case that both $E_2^-\theta$ and $E_3^-\theta$ are nonzero. Since $[E_1^-,E_3^-]=0$, if it were the case that $E_2^-(E_3^-\theta)=0$,
we would obtain that $0\neq E_3^-\theta \in \Cm_{m-1}^{l-1}$. But then, since $\theta\in U^\perp$ and $F_{m-1}E_3^-\theta\in U$, we would have that
\[ 
0 = \lpi\theta|F_{m-1}E_3^-\theta\rpi = \lpi E_1^-E_2^- \theta | E_3^-\theta\rpi - \lambda_{m-1} \lpi E_3^-\theta | E_3^-\theta\rpi = -\lambda_{m}\lpi E_3^-\theta | E_3^-\theta\rpi\neq 0,
\]
in view of Assumption \ref{a:assumption}. Hence, the existence of $0\neq \theta\in U^\perp$ implies that $E_2^-E_3^-\theta\neq 0$. But since $E_1^-(E_2^-E_3^-\theta) = [E_1^-,E_2^-]E_3^-\theta = 0$,  we would have that $0\neq E_2^-E_3^-\theta \in \Cm_{m}^{l-2}$, from which
\[0 < \lpi E_2^-E_3^-\theta | E_2^-E_3^-\theta\rpi = -\lpi \theta| E_2^+E_3^+(E_2^-E_3^-\theta)\rpi  = 0,\]
as $E_2^+E_3^+(E_2^-E_3^-\theta) \in U$. This contradiction finishes the proof.
\end{proof}

\begin{corollary}\label{c:LWdecomp}
The $\fg $-module $K_c(\tau)$ is generated by its lowest weight vectors. Moreover, it decomposes as a direct sum of irreducible lowest weight $\fg$-modules.
\end{corollary}

\begin{proof}
The claim that $K_c(\tau)$ is generated by $\Cm_c(\tau)$ is immediate from Propositions \ref{p:rowdecomps} and 
\ref{p:e1harm}. As for the second claim, from the orthogonality of the pairing, described in Proposition \ref{p:orthogonality},
we have that if $0\neq\theta\in\Cm^{l}_{m}(\sigma)$ and $0\neq\theta'\in\Cm^{l'}_{m'}(\sigma')$,
then $\Cu(\fg)(\theta)$ is orthogonal to  $\Cu(\fg)(\theta')$. Hence if $L_m^l(\sigma)$ denotes the $\fg$-module
generated by $\Cm^l_m(\sigma)$, we obtain an orthogonal direct sum decomposition $K = \oplus_{m,l,\sigma} L_m^l(\sigma)$. Finally, if $0\neq \theta \in \Cm^{l}_{m}(\sigma)$ and $L =\Cu(\fg)(\theta)$, then the exact sequence
\[ 0 \to L \to L_m^l(\sigma) \to Q \to 0, \]
where $Q$ is the quotient module, splits because of unitarity. This implies the desired result.
\end{proof}

\subsection{Refining the decomposition} 
The next task is to describe the lowest weight $\fs\fp\fo(2|2,\bbc)$-modules occurring in decomposition of Corollary 
\ref{c:LWdecomp}. We start by describing the possible bidegrees $(m,l)$ for which $\Cm_m^l\neq 0$. 
Recall that $\Ca$ denotes
the subspace of $K$ killed by $\fu^- = \bbc E_1^- \oplus \bbc E_3^-$. Let $\gls{GL}\subseteq\fg$ denote the subalgebra
generated by $H,Z,E_2^\pm$. This algebra shares the same Cartan subalgebra as $\fg$ and $\fl\cong \fg\fl(1|1)$. 
Note that:

\begin{proposition}\label{p:13harmonic}
The space $\Ca$ is an $\fl$-module. 
\end{proposition}

\begin{proof}
Certainly, $H$ and $Z$ preserve $\Ca$. If $\theta\in\Ca$, then $E_1^-(E_2^\pm \theta) = [E_1^-,E_2^\pm]\theta = 0$ 
and, similarly, $E_3^-(E_2^\pm \theta) =\{E_3^-,E_2^\pm\}\theta = 0$, from the relations in (\ref{e:evenaction}).
\end{proof}

Describing the spaces $\Cm_m^l$ is equivalent to describing the kernel of $E_2^-$ inside $\Ca_m^l$, for each bidegree
$(m,l)$. Now, for each $n\in\bbz$, define the diagonal complex
\[\gls{AHarmHom} := \bigoplus_{m+l = n} \Ca^l_m.\]
The restriction to the $\sigma$-isotypic components will be denoted by $\gls{AHarmHomIso}$, and we have 
$\Ca(n) = \oplus_\sigma \Ca(n,\sigma)$. Note that  $E_2^{\pm}(\Ca(n,\sigma))\subseteq \Ca(n,\sigma)$, from which we conclude that $(\Ca(n,\sigma),E_2^\pm)$ are chain complexes. Note also that
\[\Ca = \bigoplus_{0\leq n,\sigma\in \widehat{W}}\Ca(n,\sigma)\]
and that $\Ca(0) = K^0_0\cong V(\tau)$.

\begin{proposition}\label{p:exactcplxs}
For $n>0$, the chain complexes $(\Ca(n,\sigma),E_2^\pm)$ are exact, for all $\sigma\in\widehat{W}$. Moreover,
\[\Ca(n) = E_2^+(\Ca(n))\oplus E_2^-(\Ca(n)).\]
\end{proposition}

\begin{proof}
The first claim follows from the fact that $\{E_2^+,E_2^-\} = H+Z$ acts on $\Ca(n,\sigma)$ by the scalar
$n - (N_c(\tau)-N_c(\sigma))$, which is non-zero, because of Assumption \ref{a:assumption}. From exactness, 
the second claim follows because $\ker E_2^\pm = E_2^\pm(\Ca(n))$ from which its orthogonal complement is $E_2^\mp(\Ca(n))$.
\end{proof}

\begin{proposition}\label{p:LWdescription}
Recall that $r=\dim\fh$ and suppose that $0\neq \theta \in \Cm_m^l\subseteq \Ca_m^l$. Then,
\begin{itemize}
\item[(1)] $m=0$ implies $l=0$,
\item[(2)] $l\leq r -1$,
\item[(3)] $l = r-1$ implies $m \leq 1$.
\end{itemize}
\end{proposition}

\begin{proof}
From Proposition \ref{p:exactcplxs}, if $n=m+l>0$, then $\theta\in\Cm_m^l$ implies $\theta \in E_2^-(\Ca(n))$.
Suppose first that $m = 0$. If it were the case that $l>0$. Then, 
$\theta \in E_2^-(\Ca(l))$. But, $E_2^+\theta\in\Ca_{-1}^{l+1} = 0$, which would imply that 
$\theta \in E_2^+(\Ca(l))\cap E_2^-(\Ca(l)) =0$. Thus, $l=0$, and we proved (1). 

Assume from now that $m>0$. Then, $\theta\in E_2^-(\Ca(m+l))$, from which $l\leq r-1$, settling (2). Finally, suppose that 
$l = r-1$. Note that $E_3^+\theta \neq 0$, otherwise $0= \{E_3^+,E_3^-\}\theta = (m+1-(N_c(\tau)+N_c(\sigma)))\theta$,
which is non-zero, from Assumption \ref{a:assumption}. Then, $E_2^+E_3^+\theta\in\Ca_m^{r+1} = 0$, which implies that $E_3^+\theta = E_2^+\theta_1$, for
some $0\neq\theta_1\in \Ca$, since $E_3^+\theta\in \Ca_{m+1}^r$ and the complex $(\Ca(m+r+1),E_2^+)$ is exact. 
But then,
\[
0 < \lpi E_3^+\theta | E_3^+\theta \rpi = \lpi E_2^+\theta_1 | E_3^+\theta \rpi = \lpi E_3^-E_2^+\theta_1 | \theta \rpi =0,
\]
as $E_3^-E_2^+ = -2E_1^- + E_2^+E_3^-$. Thus, $m\leq 1$.
\end{proof}

Denote by $B(r)$, the subset of $\bbz_{\geq 0}\times\bbz_{\geq 0}$ that satisfies the conditions (1) - (3) of Proposition \ref{p:LWdescription}. 
Note that, for $r>1$ we have
\begin{equation}\label{e:LwBidegrees}
\gls{Br} := \{(0,0),(1,r-1)\}\cup\{(m,l)\mid m\geq 1,0\leq l\leq r-2\}\subseteq \bbz_{\geq 0}\times\bbz_{\geq 0}
\end{equation}
and  $B(1) = \{(0,0),(1,0)\}$. We can interpret the set $B(r)$ in terms of hook partitions. Recall that a partition is a
sequence of  integers $\lambda = (\lambda_1,\ldots,\lambda_k)$ with $\lambda_1\geq\cdots\geq\lambda_k\geq 0$. Recall
also that a partition $\lambda$ is called a {\it hook partition} if $\lambda_j \leq 1$ for $j > 1$. We shall denote
by $\lambda'$ the transpose partition $\lambda' = (\lambda_1',\ldots,\lambda_t')$, where $t = \lambda_1$ and
$\lambda'_j$ is the number of indices $i$ for which $\lambda_i\geq j$. Denote by $\gls{P}$ the set of all hook partitions. 
We shall denote by $(0)$ the empty partition and $(1^k)$ the partition $(1,\ldots,1)$ with $k\in\bbz_{> 0}$ parts equal to $1$.

\begin{proposition}
The map $\Lambda:\bbz_{\geq 0}\times \bbz_{\geq 0}\to \Cp(1|1)$ defined by 
$\Lambda(m,l)=(m,1^l)$
induces a bijection between $B(r)$ and the set
\begin{equation}\label{e:hpartition}
\gls{Hr} := \{\lambda\in \Cp(1|1)\mid \lambda_1'+\lambda_2'\leq r\}.
\end{equation}
\end{proposition}

\begin{proof}
Note that $\Lambda(0,0) = (0)$. If $(0,0)\neq(m,l)\in B(r)$, from the description (\ref{e:LwBidegrees}), we have
\[
\Lambda(m,l)'_1 + \Lambda(m,l)'_2 = \left\{
\begin{array}{ll}
l + 1 \leq r,& \textup{if } m =1\\
l + 2 \leq r,& \textup{if } m >1,
\end{array}
\right.
\]
from which $\Lambda(B(r))\subseteq H(r)$. Define a map $\Upsilon:\Cp(1|1)\to\bbz_{\geq 0}\times\bbz_{\geq 0}$ in the opposite direction
by $\Upsilon(0) =(0,0)$ and $\Upsilon(\lambda) = (\lambda_1,\lambda_1' - 1)$, if $\lambda\neq(0)$. For all $\lambda\in H(r)\setminus\{(0)\}$, from 
$\lambda_1'+\lambda_2'\leq r$ we get
\[\lambda_1' - 1\leq r-1-\lambda_2'\leq r-1.\]
Since $0\leq\lambda_2'\leq 1$, the extremal case $\lambda_1' - 1=r-1$ implies $\lambda_2'=0$ and hence $\lambda_1 = 1$. It follows
that $\Upsilon(H(r))\subseteq B(r)$. It is straightforward to check $\Upsilon\circ\Lambda|_{B(r)} = id_{B(r)}$ and 
$\Lambda\circ\Upsilon|_{H(r)} = id_{H(r)}$.
\end{proof}

\begin{theorem}\label{t:main}
Fix $c$ as in Assumption \ref{a:assumption} and let $H(r)$ be as in (\ref{e:hpartition}). For each $\lambda\in H(r)$, let
$m(\lambda) = \lambda_1$, $l(\lambda) = \lambda_1'-1$ and , for every irreducible $\bbc W$-module $\sigma$, let $L(\lambda^{\natural_\sigma})$ be an irreducible lowest weight $\fg$-module
with lowest weight $\lambda^{\natural_\sigma} = \lambda_{m(\lambda),l(\lambda),\sigma}$, defined by (\ref{e:weights}).
Then, the space $K_c(\tau)$ decomposes as a direct sum
\[
K_c(\tau) = \bigoplus_{\lambda\in H(r)} \bigoplus_{\sigma\in\widehat{W}} \dim(\Cm_{m(\lambda)}^{l(\lambda)}(\sigma))L(\lambda^{\natural_\sigma}).
\]
Moreover, in the terms of Definition \ref{d:LWmodules}, we have $L((0)^{\natural_\sigma}) = L_0(\lambda_{0,0,\sigma})$, $L((1^r)^{\natural_\sigma}) = L_1(\lambda_{1,r-1,\sigma})$ 
and $L(\lambda^{\natural_\sigma})=\bbv(\lambda^{\natural_\sigma})$,
otherwise.
\end{theorem}

\begin{proof}
The decomposition itself is a rewriting of Corollary \ref{c:LWdecomp}. As for the second part of the statement, we start 
by noting that evidently, $E_1^+\in\fg$ acts injectively on $K_c(\tau)$. Moreover, we claim that the $E_2^+$-homology of $K_c(\tau)$ is concentrated
in bidegree $(0,0)$, whereas the $E_3^+$-homology is in degree $(0,r)$. Indeed, for each $n\in\bbz$, consider the diagonal complexes 
\[K(n) := \bigoplus_{m+l = n} K^l_m\qquad\textup{and}\qquad K'(n) := \bigoplus_{m-l = n} K^l_m,\]
whose $\sigma$-isotypic components we denote by $K(n,\sigma)$ and $K'(n,\sigma)$, respectively. We note that
\[K = \bigoplus_{0\leq n,\sigma\in \widehat{W}}K(n,\sigma) =  \bigoplus_{-r\leq n,\sigma\in \widehat{W}}K'(n,\sigma).\]
Similar to the discussion for $(\Ca(n,\sigma),E_2^+)$ done in Proposition \ref{p:13harmonic},
we have that $(K(n,\sigma),E_2^+)$ and $(K'(n,\sigma),E_3^+)$ are chain complexes. For $n>0$, the complexes $(K(n,\sigma),E_2^+)$ are exact, 
since $\{E_2^+,E_2^-\} = H+Z$ acts by the nonzero scalar $n - (N_c(\tau) - N_c(\sigma))$ on $K(n,\sigma)$ and $H(K,E_2^+) = K^0_0\cong V(\tau)$. 
Similarly, since $\{E_3^+,E_3^-\} = H-Z$ acts by the scalar $r + n - (N_c(\tau) + N_c(\sigma))$ on $K'(n,\sigma)$, we conclude that 
$(K'(n,\sigma),E_3^+)$ is exact for $n >-r$ and $H(K,E_3^+) = K^r_0\cong V(\sgn\otimes\tau)$. The result now follows from Proposition \ref{p:LWmodules}.
\end{proof}

\begin{remark}
When $c=0$ and $\tau=\triv$, the decomposition described above reduces to the $(1|1)$-case in \cite[Theorem 5.39]{CW}. There,
the spaces $\Cm^l_m$ become a lowest weight module for the orthogonal group.
\end{remark}

\begin{example}
The assumptions on $c$ are rather important. Consider the case in which $\dim\fh = 1$ and $W = S_2$ is the cyclic group of
order $2$ and let $\tau=\triv$. Fix $x\in \fh^*$ and $y\in \fh$ with $\lpi x,y \rpi = 1$ and let 
$\alpha := x\in \fh^*,\alpha^\vee :=2y\in\fh$. In this case, we have that 
\[K_c(\triv) = (\bbc[x]\otimes\bbc_+)\oplus(\bbc[x]\otimes\bbc_-),\]
where $\bbc_\pm$ are the trivial and the sign representations of $S_2$. We will write $\bbc_+ = \bbc1$ and $\bbc_- = \bbc x$.
It is not hard to show that the only vectors annihilated by $E_1^-,E_2^-$ and $E_3^-$ in this case are 
$v_+:=1\otimes 1$ and $v_-:=x\otimes 1$. Whenever $c$ is as in the as in Assumption \ref{a:assumption}, we get, from Theorem \ref{t:main}
and from the fact that $\bbc S_2$ is the centraliser of $\fg=\fs\fp\fo(2|2,\bbc)$, an $(S_2,\fg)$ decomposition
\begin{equation}\label{e:rank1decomposition}
K_c(\triv) = (\bbc_+\boxtimes L_+) \oplus (\bbc_-\boxtimes L_-).
\end{equation}
Here, $L_\pm$ are lowest weight $\fs\fp\fo(2|2,\bbc)$-representation generated by $v_\pm$ and of lowest weight $\lambda_\pm$ that satisfy 
\[
(\lambda_+(H),\lambda_+(Z) )= (\tfrac{1}{2} - c,-\tfrac{1}{2} + c)\qquad\textup{and}\qquad 
(\lambda_-(H),\lambda_-(Z) ) =(\tfrac{3}{2} - c,-\tfrac{1}{2} - c).
\]
Moreover, in terms of Grassmann algebra modules (see Section \ref{s:grassmann} and Definition \ref{d:LWmodules}), we have that $L_+ \cong L(+\infty,0)$ and $L_- \cong L(+\infty,1)$. Now, note also that
\[E_2^+(v_-) = E_2^+(x\otimes 1) = (1-2c)(1\otimes x).\]
When $c=\tfrac{1}{2}$, which is outside the allowed set from Assumption \ref{a:assumption}, the module $L_-$ is also isomorphic to $L(+\infty,0)$ as a Grassmann algebra module
and the vector $1\otimes x$ will not be present the right-hand side of (\ref{e:rank1decomposition}); in this case,  
$K_c(\triv)$ does not admit a decomposition as a direct sum of lowest weight modules for $\fs\fp\fo(2|2,\bbc)$.
\end{example}

Given $\lambda\in H(r)$, or equivalently, a bidegree $(m,l)\in B(r)$, it is not {\it a priori} clear if $\Cm_m^l\neq 0$. However, in \cite[Section 5.3.5]{CW}, 
Cheng and Wang exhibited precise nonzero vectors in $\Cm^l_m$ when $\tau = \triv$ and $c=0$, for all $(m,l)\in B(r)$. Using them, it is easy to conclude that $\Cm^l_m\neq 0$ for $c=0$ and any $\tau\in\widehat{W}$. We can, on the other hand, compare the spaces of lowest weight vectors at the parameter $c$ with the ones at $0\in\mathscr{C}$. From Assumption 
 \ref{a:assumption}, following \cite[Theorem 2.9]{DO}, there is a unique $W$-equivariant linear map 
 $\vartheta:K_c(\tau)\to K_c(\tau)$ satisfying
\begin{itemize}
\item[(0)] $\vartheta(K_m^l)\subseteq K_m^l$ for all $m$ and $l$,
\item[(1)] $\vartheta$ is the identity on $K_0^0\cong V(\tau)$,
\item[(2)] $\vartheta(p\otimes \omega) = (\vartheta p)\otimes\omega$, for all $\omega\in\bbs$ and
\item[(3)] $E_2^+\vartheta = \vartheta E_2^+(0)$,
where $E_2^+(0)$ acts on $K(\tau)$ as the operator
\[E_2^+(0) = \sum_j \partial_{y_j}\otimes 1\otimes \mu_{x_j}: \bbc[\fh]\otimes V(\tau)\otimes\bbs \to\bbc[\fh]\otimes V(\tau)\otimes\bbs. \]
\end{itemize}

\begin{proposition}\label{p:comparison}
The complexes $\Ca_0(\tau,n,\sigma)$ and $\Ca_c(\tau,n,\sigma)$ are isomorphic, for all $n\in\bbz_{\geq 0}$ and $\sigma\in\widehat{W}$. In particular, $\Cm_0(\tau) \cong \Cm_c(\tau)$.
\end{proposition}

\begin{proof}
Let $E_1^-(0)$ and $E_3^-(0)$ denote the endomorphisms of $K_c(\tau)$ in which the action of $y\in \fh$ is by means of
directional derivatives in the Weyl algebra instead of their Dunkl version. Note that the intertwiner map $\vartheta$ also satisfies
$T_y(\vartheta p) = \vartheta (\partial_y p)$, for all $p\in\bbc[\fh]$ \cite[Corollary 2.13]{DO} so it follows that 
$E_1^-\vartheta = \vartheta E_1^-(0)$. Moreover, from
property (2) of $\vartheta$, we see that $\vartheta$ does not act on the spinor part. If $\star$ denotes the Hodge star operator
on $\bbs$, it is straightforward to check that $E_3^- = (-1)^p\star^{-1}E_2^+\star$ on $\bbs^p$, from which we conclude that
$E_3^-\vartheta = \vartheta E_3^-(0)$, and hence $\Ca_0(\tau,n,\sigma)\cong\Ca_c(\tau,n,\sigma)$, for all $n$.
\end{proof}

\section{The centraliser algebra of $\fs=\fs\fl(2,\bbc)$ in $\cha$}

\subsection{The centraliser algebra of $\fs\fl(2,\bbc)$}\label{s:sl2centralizer}
Since the algebra $\gls{SL}$ spanned by $H,E_1^\pm$ is isomorphic to $\fs\fl(2,\bbc)$ and is contained in $\fg$, we start by studying the
centraliser of $\fs$ inside $\cha = \cha\otimes 1\subseteq \cha\otimes\Cc$.  Recall that we have fixed an orthonormal basis 
$\{y_i\}$ of $\fh$ and a dual basis $\{x_j\}$ of $\fh^*$. Define the elements
\begin{equation}\label{e:o}
\gls{Xgen}:= x_iy_j - x_jy_i\in \cha,
\end{equation}
for $1\leq i < j \leq r$. We may extend this definition to arbitrary $i,j$, by setting $\Cx_{ji}=-\Cx_{ij}$. Denote by
$\gls{SO}$ the $\bbc$-linear span of all the elements $\Cx_{ij}$ with $1\leq i<j\leq r=\dim\fh$. 
\begin{lemma}\label{l:sl2comm}
The elements of $\fa=\textup{span}\{\Cx_{ij}\mid 1\leq i<j \leq r\}$ commute with $\fs=\fs\fl(2,\bbc)$. 
\end{lemma}

\begin{proof}
From $[H,x]=x$ and $[H,y]=-y$ for all
$x\in\fh^*$ and $y\in\fh$, it follows that each $\Cx_{ij}$ commutes with $H$. Further, let 
$p:=\sum_ux_u^2$ and $q:=\sum_uy_u^2$. From Lemma \ref{l:Hamiltonian}, we have
\begin{align*}
\!\;[p,\Cx_{ij}] &= ([p,x_iy_j]-[p,x_jy_i])\\
	&=(-2x_ix_j+2x_jx_i) =0
\end{align*}
and similarly $[q,\Cx_{ij}]=0$, finishing the proof.
\end{proof}

\begin{remark}
In the $c=0$ case, it is straightforward to check that $\fa$ is a Lie algebra isomorphic
to the orthogonal Lie algebra $\fs\fo(r,\bbc)=\fs\fo(\fh)$. For general $c$, we only have a linear isomorphism between $\fa$ and $\fs\fo(\fh)$, see Proposition \ref{p:deformedLieRelation} for the commutation relations in $\fa$.
\end{remark}

\begin{definition}
Define $\Cu=\gls{CentSL}$ as the associative subalgebra of $\cha$ generated by $\bbc W$ and $\fa$.
\end{definition}

Recall that when $c=0$ the rational Cherednik algebra becomes the smash product 
$\Cw\# W$, where $\gls{WeylAlg}$ denote the Weyl algebra acting on $\bbc[\fh]$. This algebra comes equipped with
a natural filtration in which the elements of $\fh\oplus\fh^*$ have degree one and the associated graded algebra
$\overline{\Cw}$ is isomorphic to $\bbc[\fh]\otimes\bbc[\fh^*]$. In Section \ref{s:appendix}, we show that the analogous 
subalgebras $\fs,\fa\subseteq \Cw$ are such that $\gls{ClassCentSO}$, the associative subalgebra generated by $\fa$, 
is the centralizer of $\fs$ in $\Cw$. Assuming this, we shall describe how the result in the Dunkl-Cherednik case
 can be obtained from the classical one. Let $\bt$ be an indeterminate. 
Define the generic Weyl algebra $\gls{WeylGen}$ to be the unital associative algebra over $\bbc[\bt]$ generated by 
$x\in\fh^*$ and $y\in\fh$ subject to the relations $[\bt,x]=0=[\bt,y]=[x,x']=[y,y']$ and 
\[
[y,x] = \bt\lpi x,y \rpi,
\]
for all $x,x'\in\fh^*$ and $y,y'\in\fh$. Then, the set 
$\{ \bt^nx^\alpha y^\beta\mid n\in\bbn,\alpha,\beta\in\bbn^r\}$ is a $\bbc$-linear basis for $\Cw[\bt]$. 
Here, $x^\alpha$ and $y^\beta$ is the usual multi-index notation. Define the degree $d\in\bbn$ space $\Cw[\bt]_d$
to be the linear span of the monomials  $\{ \bt^nx^\alpha y^\beta\mid 2n + |\alpha|+|\beta| = d\}$. 
It is straight-forward to check that the generic Weyl algebra $\Cw[\bt] = \oplus_{d\in\bbn} \Cw[\bt]_d$ 
is a graded $\bbc$-algebra. We note that the specialization $\Cw(1) = \Cw[\bt]\otimes_{\bbc[\bt]}\bbc_1$, in which
$\bbc_1$ is the $\bbc[\bt]$-module in which $\bt$ acts as $1$, is naturally isomorphic to $\Cw$.

The subspaces $\fs$ and $\fa$ of $\cha$ defined above have analogues in $\Cw$ and $\Cw[\bt]$ that we shall denote
in the same way. Note that the Lie subalgebra $\fs$, when seen as $\fs\subseteq \Cw[\bt]_2$ is no longer a Lie algebra 
because of the gradings involved. But we are still interested in determining the centralizer of $\fs$ in $\Cw[\bt]$. Let
 $A(\fa)[\bt]$ denote the associative subalgebra  
of $\Cw[\bt]$ generated by $\fa\subseteq\Cw[\bt]_2$. 

\begin{proposition}\label{p:gradedcent}
The subalgebra $A(\fa)[\bt]$ is the centralizer algebra of $\fs$ in $\Cw[\bt]$. As a $\bbc$-vector space, we have
\[
A(\fa)[\bt] = \bigoplus_{n\in\bbn} \bt^nA(\fa).
\]
\end{proposition}

\begin{proof}
The description of $A(\fa)[\bt] $ as $\bbc$-vector space is rather clear. Let $[\cdot,\cdot]_\Cw$ denote the commutator in 
$\Cw$ and $[\cdot,\cdot]$ is the commutator in $\Cw[\bt]$. If $v\in\fh^*\oplus\fh$ and $q\in \Cw$, 
it is straight-forward to check that $[v,q] = \bt[v,q]_\Cw$, 
from which $[v_1v_2,q] = \bt[v_1v_2,q]_\Cw$,
for all $v_1,v_2\in \fh^*\oplus\fh$, because of the Leibniz rule. It follows that, for all $X\in\fs$
and $q\in \Cw$, we have $[X,q] = \bt[X,q]_\Cw$. Thus, if $q\in\Cw$ and $p = \bt^nq$ is such that 
$[X,p]=0$ for all $X\in\fs$, it follows from
Proposition \ref{p:centsl2} in Section \ref{s:appendix}, that $q\in A(\fa)$. Next, for each $n,d\in \bbn$
let $q_n\in\Cw$ be a homogenous element of degree $d$, that is, a lift of an element of degree $d$ in 
$\overline{\Cw}$ to $\Cw$. Given $p = \sum\bt^nq_n\in\Cw[\bt]$ with $q_n$ as above, if $p$ commutes with $\fs$ 
it follows that each $\bt^nq_n\in\Cw[\bt]_{2n+d}$ commutes with $\fs$, since they all lie
in different pieces of the graded algebra $\Cw[\bt]$ and thus $q_n\in A(\fa)$. 
Since a general $p\in\Cw[\bt]$ can be expressed as finite sum of element $\bt^nq_n$ 
with $q_n\in\Cw$ homogeneous, we are done.
\end{proof}

\begin{theorem}\label{t:centralisersl2}
The algebra $\Cu$ is the centraliser of $\fs=\fs\fl(2,\bbc)$ inside $\cha$.
\end{theorem}

\begin{proof}
We know that $\Cu\subseteq \Cent_{\cha}(\fs)$. To prove the other inclusion, let $\bt$ be an 
indeterminate that commutes with $\fh^*\oplus\fh$ and $\bbc W$. Consider the generic algebra 
$\gls{RCAGen}$ defined over $\bbc[\bt]$ with  relations similar to (\ref{e:relations}), but with
\begin{equation}\label{e:genericRCArel}
[y,x] = \bt\lpi x,y \rpi - \sum_{\alpha>0} c_\alpha\lpi \alpha,y \rpi\lpi x,\alpha^\vee\rpi s_\alpha,
\end{equation}
for $y\in\fh$ and $x\in\fh^*$. Using (\ref{e:genericRCArel}), it is not hard to show (see for example\cite{CDM}, Propositions 2.5 and 2.6)
that if $p\in\bbc[\fh]$ and $q\in\bbc[\fh^*]$, then
\begin{equation}\label{e:genericpoly}
\begin{array}{rcl}
\;\![y,p] &=& \bt\partial_y(p) - \sum_{\alpha>0} c_\alpha\lpi\alpha,y\rpi\Delta_\alpha(p)s_\alpha\\
\;\![q,x] &=& \bt\partial_x(q) - \sum_{\alpha>0} c_\alpha\lpi x,\alpha^\vee\rpi\Delta_{\alpha^\vee}(q)s_\alpha,.
\end{array}
\end{equation}
where $\Delta_\alpha$ is the  operator $\alpha^{-1}(1-s_\alpha)$ acting on $\bbc[\fh]$ and $\Delta_{\alpha^\vee}$ is similarly
defined and acts on $\bbc[\fh^*]$.

Note that as a $\bbc$-vector space, we have 
$\cha[\bt] = \bbc[\bt]\otimes \bbc[\fh]\otimes\bbc[\fh^*]\otimes\bbc W$. Consider the filtration
of $\cha[\bt]$ in which $\bt$ is in degree $2$, $\fh^*\oplus\fh$ is in degree $1$ and $\bbc W$ is in 
degree $0$. Then the associated graded $\bbc$-algebra, $\textsf{Gr}(\cha[\bt])$,  is isomorphic to $\Cw[\bt]\# W$.
If we let $\Cu[\bt]$ denote the associative subalgebra of $\cha[\bt]$ generated by $\fa$ and $\bbc W$ and endow
$\Cu[\bt]$ with the filtration of $\cha[\bt]$ described above, then $\textsf{Gr}(\Cu[\bt])\cong A(\fa)[\bt]\#W$. Further, from 
(\ref{e:genericpoly}), if $p\in\bbc[\fh]\otimes\bbc[\fh^*]\subseteq\cha[\bt]$ and $X\in\fs$, then 
\begin{equation}\label{e:comm}
[X,p]_{\cha[\bt]} \equiv [X,p]_{\Cw[\bt]},
\end{equation}
modulo lower filtration degree. Thus,  it follows from (\ref{e:comm}), Proposition \ref{p:gradedcent} and 
$\textsf{Gr}(\Cu[\bt])\cong A(\fa)[\bt]\#W$, that 
$\Cent_{\cha[\bt]}(\fs) \subseteq \Cu[\bt]$. Specializing to $\bt=1$ gives the desired result.
\end{proof}

For each $0\leq l \leq r=\dim\fh$, the subspace $K_c(\tau)^l = \oplus_{m\geq 0} K^l_m$ is an $\cha$-module isomorphic to $M_c(\tau\otimes\wedge^l\fh)$. 
The $(W,\fs)$ decomposition these modules were essentially studied in \cite{BSO}. Since the algebra $\fs$ is $\ast$-invariant, then so is its centraliser $\Cu$. We have:

\begin{proposition}
For each $m,l\in\bbz_{\geq 0}$ the space $\Ch_m^l$ is a $\ast$-unitary $\Cu$-module. Moreover, when $l=0$, as an $(\Cu,\fs)$-module we have
\begin{equation}\label{e:sl2isotypic}
K_c(\tau)^0 = \bigoplus_{m\in\bbz_{\geq 0}}  \Ch_c(\tau)_m^0\boxtimes L_{\fs\fl(2)}(m + r/2 - N_c(\tau)) ,
\end{equation}
where $L_{\fs\fl(2)}(\mu)$ denotes the irreducible lowest weight $\fs\fl(2)$-module with lowest weight $\mu\in\bbc$.
\end{proposition}

\begin{proof}
Since $\Cu$ is the centraliser of $\fs = \fs\fl(2,\bbc)$, each $\Ch_m^l$ is a $\Cu$-module, and the unitarity is because $K_c(\tau)$ is a unitary $\cha\otimes\Cc$-module. The decomposition follows from Theorem \ref{t:main}; note that when $l=0$, we have that
$E_2^-$ and $E_3^-$ acts as zero on $K_m^l$, from which $\Cm_m^0 = \Ch_m^0$.  
\end{proof}

\subsection{Commutation relations} \label{s:6.2}
We now discuss the commutators of the elements $\Cx_{ij}$. In fact, the description of $\Cu_c$ 
in terms of generators and relations has already been obtained in \cite{FeHa}. If $x\in\fh^*$ and $y\in \fh$, let
$\gls{Xgen2} = xy-y^*x^*$ so that $\fa$ is the linear span of $\{\Cx_{xy}\mid x\in\fh^*,y\in \fh\}$. Recall
that $\fh = E_\bbc$, the complexification of the real Euclidean space $E$.

\begin{proposition}
In $\cha$, we have $\Cx_{xy}^* = -\Cx_{xy}$ for all $x\in\fh^*$ and $y\in\fh$. It also holds, for $x,\xi\in E^*\subseteq \fh^*$ and $y,\eta\in E\subseteq \fh$ that
\begin{align}\label{e:comm-X}
[\Cx_{xy},\Cx_{\xi\eta}] &=\Cx_{x\eta}[y,\xi]-\Cx_{\xi y}[\eta,x] -\Cx_{y^*\eta}[x^*,\xi]+\Cx_{\xi x^*}[\eta,y^*]\\\label{e:quad2} 
\Cx_{xy}\Cx_{\xi\eta} &= \Cx_{\xi y}\Cx_{x\eta}  + \Cx_{x\xi^*}\Cx_{y^*\eta} + \Cx_{xy}[\eta,\xi] - \Cx_{\xi y}[\eta,x]
- \Cx_{x\xi^*}[\eta,y^*].
\end{align}
\end{proposition}

\begin{remark}\label{r:usual}
Notice that the ``structure constants" of the $\Cx$'s in the commutation relation (\ref{e:comm-X})  are elements of $\mathbb C W$. When $c=0$ and we let $x=x_i,y=y_j,\xi = x_k, \eta = y_k$, (\ref{e:comm-X}) becomes the commutation relation in the Lie algebra $\mathfrak{so}(\fh)$: \[[\Cx_{ij},\Cx_{kl}]_0=\delta_{jk}\Cx_{il}-\delta_{li}\Cx_{kj}-\delta_{ik}\Cx_{jl}+\delta_{lj}\Cx_{ki}.\]
\end{remark}

\begin{proof}
The equation $\Cx_{xy}^* = -\Cx_{xy}$ is immediate from the definition. Claims (\ref{e:comm-X}) and (\ref{e:quad2}) were discussed in \cite{FeHa} and can be verified with a straight-forward computation using the identities
\begin{itemize}
\item[(a)] $[\xi_1\eta_1,\xi_2\eta_2] = \xi_1[\eta_1,\xi_2]\eta_2 - \xi_2[\eta_2,\xi_1]\eta_1$,
\item[(b)] $\xi_1\eta_1\xi_2\eta_2 = \xi_2\eta_1\xi_1\eta_2 + \xi_1[\eta_1,\xi_2]\eta_2 - \xi_2[\eta_1,\xi_1]\eta_2$,
\item[(c)] $\xi_1\eta_1\xi_2\eta_2 = \xi_1\eta_2\xi_2\eta_1 + \xi_1[\eta_1,\xi_2]\eta_2 - \xi_1[\eta_2,\xi_2]\eta_1$,
\item[(d)]  $ [[\eta_1,\xi_1],\eta_2] = [[\eta_2,\xi_1],\eta_1]$
\end{itemize}
for all $\xi_1,\xi_2\in \fh^*,\eta_1,\eta_2\in \fh$ and  
\begin{itemize}
\item[(e)]  $[x^*,y^*]=[y,x]$,
\end{itemize}
 whenever $x\in E^*,y\in E$.
\end{proof}

The relations in (\ref{e:comm-X}) can be written in a coordinate-free form. Consider $T:\fh\to\fh^*$, the linear isomorphism that sends $y_i\mapsto y_i^T=x_i$,
for all $i$. We thus have an identification $\fa = \textup{span}\{\Cx_{ij}\} \cong \wedge^2\fh^*$, under which, we have that $\Cx_{ij}$ corresponds to 
$x_i\wedge x_j$. We shall  still denote by $\lpi\cdot,\cdot\rpi$ the bilinear pairing $\fh^*\times\fh^*\to\bbc$  and extend  it to a non-degenerate bilinear 
pairing $\wedge^p\fh^*\times\wedge^p\fh^*\to\bbc$, for all $p$, via the determinant: for all $\psi_i ,u_j\in\fh^*$, let
\[\lpi \psi_1\wedge\cdots\wedge\psi_p, u_1\wedge\cdots\wedge u_p \rpi = \det(\lpi \psi_i,u_j \rpi).\]
If $\Cx,\Cy\in\fa$ and $u,\psi\in\fh^*$, define the contractions
$\Cx\lcont u\in\fh^*$ and $\psi\rcont \Cy\in\fh^*$ by requiring for all $v,\phi\in\fh^*$
\[
\lpi \Cx\lcont u,v \rpi = \lpi \Cx,u\wedge v \rpi \qquad\textup{and}\qquad \lpi \phi,\psi \rcont \Cy \rpi = \lpi \phi\wedge\psi, \Cy \rpi.
\]
We note that $x\rcont \Cy = -\Cy\lcont x$ for all $x\in\fh^*$ and $\Cy\in\fa$. Note also that $\fa$ is not closed under the bracket of $\cha$, unless $c=0$. Let $[\cdot,\cdot]_0:\fa\times\fa\to\fa$ be the Lie 
bracket on $\fa\cong\fs\fo(\fh)$.  For each $\alpha\in R\subseteq\fh^*$ define the skew-bilinear map $\kappa_\alpha:\fa\times\fa\to\fa$ via

\begin{equation}\label{e:kappa}
\kappa_\alpha(\Cx,\Cy) = (\Cx\lcont \alpha^\vee)\wedge(\alpha\rcont \Cy).
\end{equation}
Since this expression is quadratic, we have $\kappa_{\alpha} = \kappa_{-\alpha}$.

\begin{proposition}\label{p:deformedLieRelation}
For any $\Cx,\Cy\in\fa\subseteq \cha$, we have
\begin{equation}\label{e:Urel}
[\Cx,\Cy] = [\Cx,\Cy]_0 - \sum_{\alpha>0}c_\alpha\kappa_{\alpha}(\Cx,\Cy)s_\alpha.
\end{equation}
\end{proposition}
\begin{proof}
By bilinearity, it suffices show that the right-hand side agrees with (\ref{e:comm-X}), when we substitute $\Cx$ and $\Cy$ 
by $\Cx_{ij}$ and $\Cx_{kl}$, respectively. For a fixed $\alpha>0$, one computes, for all $u\in \fh^*$,  that
\[
\lpi \Cx_{ij}\lcont\alpha^\vee, u \rpi = \lpi \Cx_{ij},\alpha^\vee\wedge u\rpi = \lpi x_i,\alpha^\vee \rpi \lpi x_j,u\rpi - \lpi x_j,\alpha^\vee \rpi \lpi x_i, u\rpi,
\]
from which $\Cx_{ij}\lcont\alpha^\vee = \lpi x_i,\alpha^\vee \rpi  x_j - \lpi x_j,\alpha^\vee \rpi  x_i$. Similarly, $\alpha\rcont \Cx_{kl} = \lpi\alpha,x_l\rpi x_k - \lpi\alpha,x_k\rpi x_l.$
Then,
\begin{multline}\label{e:1steq}
(\Cx_{ij}\lcont\alpha^\vee)\wedge(\alpha\rcont \Cx_{kl}) = \\
\lpi \alpha,x_j \rpi \lpi x_k,\alpha^\vee \rpi \Cx_{il} - \lpi \alpha,x_l \rpi \lpi x_i,\alpha^\vee \rpi \Cx_{kj} - \lpi \alpha,x_i \rpi \lpi x_k,\alpha^\vee \rpi \Cx_{jl} + 
\lpi \alpha,x_l \rpi \lpi x_j,\alpha^\vee \rpi \Cx_{ki}.
\end{multline}
It follows from (\ref{e:1steq}) that 
\[[\Cx_{ij},\Cx_{kl}] = [\Cx_{ij},\Cx_{kl}]_0 - \sum_{\alpha > 0} c_\alpha\kappa_{\alpha}(\Cx_{ij},\Cx_{kl})s_\alpha, \]
and we are done.
\end{proof}

\begin{theorem}[\cite{FeHa}]\label{t:genrel}
The algebra $\Cu_c$ is the quotient of the smash product algebra $T(\wedge^2\fh^*)\# W$ subject to the relations
(\ref{e:comm-X}) and (\ref{e:quad2}), under the linear isomorphism $\fh\to \fh^*$ that sends $y_i\mapsto x_i$ for all $i$.
\end{theorem}

\begin{proof}
This is Theorem 4 in \cite{FeHa} in the case of the symmetric group. Their proof generalizes to any Coxeter group $W$,
as is noted in Section 8 of \cite{FeHa}.
\end{proof}

\begin{remark}
The upshot of Theorem \ref{t:genrel} is that $\Cu_c$ is a non-homogeneous quadratic algebra of PBW type, in the sense of \cite{BG}.
\end{remark}

\begin{remark}
The algebra $\Cu_c$ is not a Drinfeld orbifold algebra, in the sense of \cite{SW}.
\end{remark}

\section{The centraliser algebra of $\fg=\mathfrak{spo}(2|2)$ in $\cha\otimes\Cc$} 
Let ${ \Cr }=\gls{CentSPO}$ denote the (super)centraliser of $\fg$ inside $\cha\otimes\Cc$. From the $W$-invariance of the elements in $\fg$, it is clear that 
$\rho(\bbc W)\subseteq { \Cr }$. The spaces $\Cm_m^l(\sigma)$ of lowest vectors for $\fg$ are
${ \Cr }$-modules. In this section, we investigate the algebra ${ \Cr }$. 

\subsection{The centraliser algebra of $\fg$} Denote by 
\[\gls{CentreDiag}=\text{ the centre of }\bbc W.
\]
From the previous subsection, we know the centraliser of the $\mathfrak {sl}(2)$ in $\cha$ is $\mathcal U$. Define the Dirac element 
\begin{equation}
D=E_2^++E_2^-\in \mathfrak g.
\end{equation}

\begin{lemma}\label{l:D-cent}
${ \Cr }=Z_{\mathcal U\otimes \mathcal C}(D)$.
\end{lemma}

\begin{proof}
Clearly ${ \Cr }\subseteq Z_{\mathcal U\otimes \mathcal C}(D)$. For the converse, it is sufficient to check that $H,E_1^+, E_1^-$ and $D$ generate $\mathfrak g$. Recall the relations (\ref{e:gl2}), (\ref{e:evenaction}) and (\ref{e:anti-commutation}) in $\mathfrak g=\mathfrak{spo}(2|2,\bbc)$ and their generators in $\cha\otimes\Cc$, given in (\ref{e:spodefs}). We have $E_2^\pm=\frac 12(D\mp [H,D])$, so $E_2^+$ and $E_2^-$ are obtained. Next since $D^2=\{E_2^+,E_2^-\}=H+Z$, $Z$ is generated. Finally, $E_3^\pm=-[E_1^\pm,E_2^\pm]$, so all of $\mathfrak g$ is generated.
\end{proof}

\begin{lemma}\label{Z-cent}
  Assume that the parameter $c$ is such that $N_c(\tau)-N_c(\sigma)\notin \mathbb Z$, for all $\tau\neq\sigma\in \widehat W$ (in particular, this is the case when $c$ is as in Assumption \ref{a:assumption}). Then:

  \begin{enumerate}
  \item[(i)] An element $h\in \cha\otimes \mathcal C$ commutes with $Z=Z_0+\Omega_c$ if and only if $[Z_0,h]=0=[\Omega_c,h]$. In particular, ${ \Cr }=Z_{\mathcal U\otimes \mathcal C_0}(D)$.
  \item[(ii)] $[\Omega_c,h]=0$  if and only if $h$ commutes with $\rho(Z_W)$. Therefore, \[{ \Cr }\subset Z_{\mathcal U\otimes \mathcal C_0}(\rho(Z_W))=\bigoplus_{\sigma\in\widehat W}e_\sigma (\mathcal U\otimes \mathcal C_0) e_\sigma,\]
  where $\{e_\sigma:\sigma\in\widehat W\}$ is the system of orthogonal idempotents in $Z_W$.
    \end{enumerate}
\end{lemma}

\begin{proof}
  Suppose $h\in \cha\otimes\mathcal C$ commutes with $Z$. Since $\operatorname{ad}(Z)$ preserves each subspace $\cha\otimes\mathcal C_n$, we may assume without loss of generality that $h\in \cha\otimes\mathcal C_n$. Recall that $\operatorname{ad}(Z_0)$ acts by $n\cdot \operatorname{Id}$ on $\cha\otimes\mathcal C_n$, hence we wish to prove that $n=0$. By the assumptions, we have $[\Omega_c,h]=-n h$. Consider the action of $h$ on the $\cha\otimes\mathcal C$-module $K_c(\triv)$. If $f$ is an element of $K_c(\triv)$ which lies in the isotypic component of $\sigma\in \widehat W$, and we denote $f'=h\cdot f$, then we find that
  \[\Omega_c\cdot f'=(N_c(\sigma)-n) f'.
  \]
  Hence $f'$ is an eigenvector for the action of $\Omega_c$, but then it should have an eigenvalue of $N_c(\tau)$ for some $\tau\in \widehat W$. From the hypothesis of the lemma, it follows that $n=0$.

  Moreover, this calculation shows that for every $\sigma\in \widehat W$, if $f$ is in the $\sigma$-isotypic component of $K_c(\triv)$, then so is $f'=h\cdot f$. But this means that the action of $h$ commutes with that of $\rho(e_\sigma)$, where  $e_\sigma\in Z_W$ is the idempotent corresponding to $\sigma$. Since $K_c(\triv)$ is a faithful $\cha\otimes \mathcal C$-module, it follows that $[h,e_\sigma]=0$ in $\cha\otimes \mathcal C$. But since $\{e_\sigma\}$ is a basis of $\bbc W^W$, the first inclusion follows. The rest is a general algebra fact: if $\{e_i\}$ is a finite system of orthogonal idempotents in an associative unital algebra $A$, such that $\sum e_i=1$, then the commutator of $\{e_i\}$ in $A$ is $\bigoplus_i e_i A e_i$, as it can be seen easily from the decomposition $A=\bigoplus_{i,j} e_i A e_j$.
\end{proof}

Since $\mathfrak g$ is closed under the anti-involution $*$, so is ${ \Cr }$. 
\begin{proposition}\label{c-unitary}
In the notation of Theorem \ref{t:main}, for each $\lambda\in H(r)$ and $\sigma\in \widehat W$, the isotypic space $\mathcal M_{m(\lambda)}^{l(\lambda)}(\sigma)$ is a $*$-unitary ${ \Cr }$-module. Moreover, we have an $(\Cr,\fg)$-decomposition
\begin{equation}\label{e:spoisotypic}
K_c(\tau,\sigma) = \bigoplus_{\lambda\in H(r)}  \Cm_{m(\lambda)}^{l(\lambda)}(\sigma) \boxtimes L(\lambda^{\natural_\sigma}).
\end{equation}
\end{proposition}

\begin{proof} 
The unitarity follows from the fact that $K_c(\tau)$ is a unitary $\cha\otimes \mathcal C$-module. Since $\Cr$ is the centraliser
of $\fg$, the $ (\Cr,\fg) $-module  decomposition as in the statement follows from Theorem \ref{t:main}.
\end{proof}

We now proceed to determine $\Cr$. For that, 
we use a variation of the idea of Dirac cohomology in the setting of Cherednik algebras, as in \cite{Ci}. Define
\[(\cha\otimes\mathcal C)_0=\bigoplus_{n\in \mathbb Z} (\cha)_n\otimes \mathcal C_{-n},
\]
to be subspace of elements of total degree $0$. Here, just as for $\cha$, we give degree $+1$ to $x$ and degree $-1$ to $y$ in $\mathcal C$ and denote by $\mathcal C_l$ the subspace of total degree $l$ in $\mathcal C$ (see Remark \ref{r:clifftotaldeg}). 
Note that the defining relations in $\mathcal C$ preserve the grading, which makes $\mathcal C$ a $\mathbb Z$-graded algebra (not just filtered). 

It is clear from the relations that $(\cha\otimes \mathcal C)_0$ is the centraliser in $\cha\otimes \mathcal C$ of $H+Z_0$. Let $(\cha\otimes \mathcal C)_0^{Z_W}$ denote the centraliser of $\rho(Z_W)$ in  $(\cha\otimes \mathcal C)_0$. Define 
\begin{equation}
\begin{aligned}
&\gls{diff}: (\cha\otimes \mathcal C)_0\to (\cha\otimes \mathcal C)_0,\quad d(a)=Da-(-1)^{|a|}D, \\
&\gls{diffW}: (\cha\otimes \mathcal C)_0^{Z_W}\to (\cha\otimes \mathcal C)_0^{Z_W},\quad d_W=d\mid_{(\cha\otimes \mathcal C)_0^{Z_W}},
\end{aligned}
\end{equation}
where $|a|$ is the parity of $a$, i.e., $|a|=(n \text{ mod } 2)$, if $a\in \cha\otimes \mathcal C_n$; in other words, $d$ is the supercommutator of $D$. It is immediate that \[d^2=[D^2,-]=[H+Z,-]\]
and therefore, $d^2_W=0$ (while $d^2\neq 0$ because of the presence of $\Omega_c$). 

\begin{theorem}[cf. {\cite[Theorem 3.5]{Ci}}]\label{t:vogan}
$\ker d_W=\im d_W\oplus \rho(\mathbb C W).$
\end{theorem}

\begin{proof}
The corresponding statement for $d$ restricted to $(\cha\otimes \mathcal C)_0^{W}$ is proved in \cite{Ci} and it is based on three steps:
\begin{enumerate}
\item A passage to the associated graded $S(\fh+\fh^*)\#W\otimes \mathcal C$ (with respect to the PBW-filtration of $\cha$), where the corresponding differential $\overline d_W$ is identified with the differential in a certain $W$-deformed Koszul complex;
\item A proof that the deformed Koszul complex is exact except in degree $0$ where the cohomology equals $\bbc W$;
\item An inductive step to go from the associated graded $\overline d_W$ to $d_W$ (\cite[\S 3.3]{Ci}), where one needs to use, in the language of this paper, that $d_W$ has total degree $0$ in $\cha\otimes\mathcal C$, which allows one to lift the appropriate elements from $(S(\fh+\fh^*)\#W\otimes \mathcal C)_0$ to $(\cha\otimes \mathcal C)_0$.
\end{enumerate}
All the steps of the proof work just the same with $Z_W$-invariants in place of $W$-invariants. Notice that $\mathbb C W^{Z_W}=\bbc W$, hence the absence of any invariants in the $0$-cohomology of $d_W$.
\end{proof}

\begin{corollary}\label{c:c^W}
${ \Cr }=\rho(\bbc W)\oplus ~(d(\cha\otimes \mathcal C))^{Z_W}\cap (\mathcal U\otimes \mathcal C_0)^{Z_W}.$ 
\end{corollary}

\begin{proof}
By Lemmas \ref{l:D-cent} and \ref{Z-cent}, ${ \Cr }=Z_{(\mathcal U\otimes \mathcal C_0)^{Z_W}}(D)$.  But an element in $(\mathcal U\otimes \mathcal C_0)^{Z_W}$ which commutes with $D$ is the same as an element of $\ker d_W\cap (\mathcal U\otimes \mathcal C_0)^{Z_W}$. This means that ${ \Cr }=\ker d_W\cap (\mathcal U\otimes \mathcal C_0)^{Z_W}$, and the claim now follows from Theorem \ref{t:vogan}, using  that $\rho(\mathbb C W)\subset (\mathcal U\otimes \mathcal C_0)^{Z_W}$ and that $d((\cha\otimes \mathcal C)^{Z_W})=(d(\cha\otimes \mathcal C))^{Z_W}$ since $d$ is $W$-invariant.
\end{proof}

The expressions $\Cx_{ij} = x_iy_j - x_jy_i$ define elements in both in $\Cc$ and in $\cha$. Define
\begin{equation}\label{e:sodeformed1}
\gls{Xdiag} = \Cx_{ij}\otimes 1+1\otimes  \Cx_{ij}\in \cha\otimes \mathcal C.
\end{equation}

\begin{remark}
As it is well known, see \cite{CW} for example, when $c=0$, the span of $\Cx^\Delta_{ij} $ defines a Lie algebra $\gls{SOdiag}$ isomorphic to $\mathfrak{so}(r,\bbc)$ inside $\Cw\otimes\Cc$, the tensor product of the Weyl algebra and the Clifford algebra, which centralises $\mathfrak{spo}(2|2,\bbc)$.
\end{remark}

For every $i,j$, define the elements
\begin{equation}\label{e:sodeformed2}
\gls{XdiagF}=\frac 12 d(x_i\otimes y_j-x_j\otimes y_i+y_j\otimes x_i-y_i\otimes x_j).
\end{equation}

\begin{lemma}\label{l:d-one}
We have $\mathcal F_{ij}^*=-\mathcal F_{ij}$ and \[\mathcal F_{ij}=\Cx_{ij}^\Delta-\frac 12\sum_{\alpha>0}c_\alpha s_\alpha\otimes (1-s_\alpha)(\Cx_{ij}).\]
In particular, $ \mathcal F_{ij}\in \mathcal U\otimes \mathcal C_0$.
\end{lemma}

\begin{proof}
The first claim follows immediately from the definition of $*$. The formula follows from a direct and easy calculation whose complete details we skip. For example, one verifies that $d(x_i\otimes y_j-x_j\otimes y_i)=\Cx_{ij}^\Delta-\sum_{\alpha>0}c_\alpha s_\alpha\otimes \alpha (\langle x_i,\alpha^\vee\rangle y_j-\langle x_j,\alpha^\vee\rangle y_i)$, and the second half of $\mathcal F_{ij}$ follows by applying $*$. The final claim is then clear from the definition of $\mathcal U$.
\end{proof}

Motivated by Corollary \ref{c:c^W} and Lemma \ref{l:d-one}, we make the following definition.
\begin{definition}
Let $\mathcal V=\gls{Valg}$ be the subalgebra of $\cha\otimes \mathcal C$ generated by the elements $\mathcal F_{ij}$ and $\bbc W$.
\end{definition}

\begin{remark}
When $c=0$ the algebra $\mathcal V$ is isomorphic to $A(\fa^\Delta)\#W$, where $A(\fa^\Delta)$ is the associative subalgebra of $\Cw\otimes\Cc$ generated by $\fa^\Delta\cong\mathfrak{so}(r,\bbc)$,
the diagonal copy of $\fs\fo(r,\bbc)$ in $\Cw\otimes\Cc$. This algebra is the centraliser algebra of $\fs\fp\fo(2|2,\bbc)$
 in $\Cw\otimes\Cc$. In Section \ref{s:appendix}, we provide a proof of this statement. In general, 
$\mathcal V$ can be regarded as a deformation of the smash product 
$A(\fa^\Delta)\#W$.
\end{remark}

 It is clear from the discussion so far that, 
\begin{equation}
\mathcal V_{c}^{Z_W}\subseteq { \Cr }.
\end{equation}

In order to argue about the reverse  inclusion, let again $\bt$ be an indeterminate
and consider the generic algebra $\cha[\bt]$ defined over $\bbc[\bt]$ used in the proof
of Theorem \ref{t:centralisersl2}. With a similar computation to the one in Lemma \ref{l:d-one}
but now in $\cha[\bt]\otimes\Cc$, the element
$\Cf_{ij}$ defined in (\ref{e:sodeformed2}) satisfies
\begin{equation}\label{e:Fgeneric}
\Cf_{ij}=\Cx_{ij}\otimes 1 + \bt\otimes\Cx_{ij}
-\frac 12\sum_{\alpha>0}c_\alpha s_\alpha\otimes (1-s_\alpha)(\Cx_{ij})\in\cha[\bt]\otimes\bbc.
\end{equation}
Consider the filtration of $\cha[\bt]\otimes \mathcal C$ in which the elements of 
$\mathbb CW\otimes\mathcal C$ are in degree $0$, the elements of $V=\fh^*\oplus\fh$ are in degree one, and the 
indeterminate $\bt$ is in degree $2$. The associated graded object is 
$\textsf{Gr}(\cha[\bt]\otimes \mathcal C)=(\Cw[\bt]\# W)\otimes \mathcal C$. 
We remark that the degree $d$-space is of $(\Cw[\bt]\# W)\otimes \mathcal C$ is linearly isomorphic to 
$\Cw[\bt]_d\otimes\bbc W\otimes\Cc$.
Take the induced filtration on $\mathcal V[\bt]$, the subalgeabra of $\cha[\bt]\otimes\Cc$ analogous to $\Cv$, 
and let $\mathsf{Gr}(\mathcal V[\bt])$ be the associated graded object. Let us write
\[
\gls{XdiagG} = \Cx_{ij}\otimes 1 + \bt\otimes\Cx_{ij}
\]
and denote the linear span of $\{\Cg_{ij}\mid 1\leq i<j\leq r\}$ by $\fa^\Delta[\bt]\subseteq \Cw[\bt]_2\otimes\Cc$.
\begin{lemma}\label{V-ag}
There is an algebra isomorphism $\mathsf{Gr}(\mathcal V[\bt])\cong A(\fa^\Delta[\bt])\# W$, where
$A(\fa^\Delta[\bt])$ is the associative subalgebra of $\Cw[\bt]\otimes\Cc\subseteq (\Cw[\bt]\# W)\otimes \mathcal C$ 
generated by $\fa^\Delta[\bt]$.
\end{lemma}

\begin{proof}
From (\ref{e:Fgeneric}), we see that in 
$\mathsf{Gr}(\mathcal V[\bt])$,  $\Cf_{ij}\equiv \Cg_{ij}$. On the other hand, 
in $\cha\otimes \mathcal C$, equation (\ref{e:Urel}), together with an easy (and known) calculation of the commutator of the elements $\Cx_{ij}$ in $\mathcal C$, imply that:
\begin{equation}
[\mathcal G_{ij},\mathcal G_{kl}] = \bt[\mathcal X^\Delta_{ij},\mathcal X^\Delta_{kl}]_0-\sum_{\alpha>0}c_\alpha \kappa_a(\mathcal X_{ij},\mathcal X_{kl}) s_\alpha\otimes 1\in\cha[\bt]\otimes\Cc.
\end{equation}
Here, we denote by $[\mathcal X^\Delta_{ij},\mathcal X^\Delta_{kl}]_0$ the commutator in the Lie algebra 
$\fa^\Delta\cong\mathfrak{so}(r,\bbc)^\Delta$ inside $\Cw\otimes\Cc$. But this means that, in the associated graded algebra, $[\mathcal G_{ij},\mathcal G_{kl}]\equiv\bt[\mathcal X^\Delta_{ij},\mathcal X^\Delta_{kl}]_0$ 
and this is enough to conclude the claim of the lemma.
\end{proof}

\begin{proposition}\label{p:gradedcentspo}
The subalgebra $A(\fa^\Delta[\bt])$ is the centralizer algebra of $\fg$ in $\Cw[\bt]\otimes\Cc$. 
\end{proposition}

\begin{proof}
Let $u\in\Cw$ and suppose that $p = u\otimes c\in\Cw\otimes\Cc$. Let $[\cdot,\cdot]_{\Cw\otimes\Cc}$ denote the commutator in $\Cw\otimes\Cc$. Then, note that
\begin{equation}\label{e:commD}
[D,p] = [D,u](1\otimes c) + (u\otimes1)[D,c] = [D,u]_{\Cw\otimes\Cc}(\bt\otimes c) + (u\otimes 1)[D,c].
\end{equation}
Now let $p=\sum_i u_i\otimes c_i\in \Cw[\bt]\otimes\Cc$ be an element that commutes with $\fg$.
Then, we can assume that $u_i\in A(\fa)[\bt]$, $c_i\in\Cc_0$ and $[D,p]=0$. We can decompose an
element $u_i\in A(\fa)[\bt]$ as $u_i = \sum\bt^nq_{i,n}$ with $q_{i,n}$ homogeneous,
 for each $i$.  Further, recall the filtration on $\Cc$ in which $V=\fh^*\oplus\fh$ 
has degree one. The associated graded object satisfies $\mathsf{Gr}(\Cc) = \wedge(\fh^*\oplus\fh)$,
so we can assume that 
 $p=\sum_i u_i\otimes c_i\in \Cw[\bt]\otimes\Cc$ with $u_i = \bt^nq_i$, $n$ fixed
 and $q_i\in\Cw$, $c_i$ homogeneous. But using (\ref{e:commD}) we obtain
\[
[D,p] = \bt^n\left(\sum_i [D,q_i]_{\Cw\otimes\Cc}(\bt \otimes c_i) + (q_i\otimes1)[D,c_i]\right)=0,
\]
so it suffices to prove the result for $p=\sum_i u_i\otimes c_i$ with 
$u_i \in\Cw$ and $c_i\in\Cc_0$ homogeneous. But, in this situation, we claim that if all $c_i$ are
of positive degree, then  $[D,p]=0$. Indeed, we can adapt the arguments in the proof of 
Proposition \ref{p:centspo} to conclude that. This works because the crucial part of that argument uses
the commutation of $D$ (or rather of $E_2^+$) with the Clifford algebra part and that does not have
any factor of $\bt$ involved. 
Granted this claim, if there are $c_i$ of degree zero, we can subtract a suitable
element of $\Cg^\Delta\in A(\fa^\Delta[\bt])$ so that $p-\Cg^\Delta$ is a linear combination of $u'_i\otimes c'_i$ with $c'_i$ all of positive degree. Since we also have $[D,p-\Cg^\Delta]=0$, 
the claim implies that $p=\Cg^\Delta\in A(\fa^\Delta[\bt])$. This finishes the proof.
\end{proof}

\begin{theorem}\label{t:cent-g}
Assume that the parameter $c$ is such that $N_c(\tau)-N_c(\sigma)\notin \mathbb Z$, for all $\tau\neq\sigma\in \widehat W$. Then the centraliser algebra of $\fg$ in $\cha\otimes \mathcal C$ is ${ \Cr }=\mathcal V^{Z_W}$.
\end{theorem}

\begin{proof}
We already know that $\mathcal V^{Z_W}\subseteq { \Cr }$. We shall work on $\cha[\bt]\otimes\Cc$. 
From Corollary \ref{c:c^W}, adapted to $\cha[\bt]\otimes\Cc$, we have that ${ \Cr[\bt] }=\ker d_W\cap (\mathcal U[\bt]\otimes \mathcal C_0)^{Z_W}$. With respect to the filtration on $\cha[\bt]\otimes\Cc$ considered above, the decomposition
of Corollary \ref{c:c^W} becomes
\begin{equation*}
\mathsf{Gr}({ \Cr }[\bt])=\bbc W\oplus \bar d(\Cw[\bt]\#W\otimes \mathcal C)^{Z_W}\cap 
(\mathsf{Gr}(\Cu[\bt])\otimes \mathcal C_0)^{Z_W}.
\end{equation*}
Here, $\bar d$ is the corresponding differential in $(\Cw[\bt]\#W)\otimes \mathcal C$ and we noted in the proof of 
Theorem \ref{t:centralisersl2} that $\mathsf{Gr}(\Cu[\bt])\cong A(\fa)[\bt]\#W$. However, from Proposition \ref{p:gradedcentspo} and the fact that $\fg = \fs\fp\fo(2|2,\bbc)$ is generated by $D$ and $\fs\fl(2,\bbc)$, 
we see that in the associated graded object, we have 
\begin{equation*}
\bar d(\Cw[\bt]\#W\otimes \mathcal C) \cap  (A(\fa)[\bt]\#W)\otimes \mathcal C_0= A(\fa^\Delta[\bt])\#W.
\end{equation*}
Hence $\mathsf{Gr}({ \Cr }[\bt])= (A(\fa^\Delta[\bt])\#W)^{Z_W}$. Then, comparing with Lemma \ref{V-ag}, we get 
$\mathsf{Gr}({ \Cr }[\bt])\cong \mathsf{Gr}(\mathcal V[\bt])^{Z_W}$. The claim in the theorem follows by specializing 
$\bt=1$.
\end{proof}

\begin{remark}
Computations of the commutators $[\Cf_{ij},\Cf_{kl}]$ in $\Cv$ indicate that, unlike $\Cu$, $\Cv$ is not a (non-homogeneous) quadratic algebra over $\bbc W$, or at least not with the generators $\Cf_{ij}$. However, one can easily see that $\Cv\otimes\Cc$  is a quadratic algebra over $\bbc W\otimes \Cc$.
\end{remark}

\subsection{Dual pairs} 

Here, we apply the theory of dual pairs in order to study some of the multiplicities of the decompositions obtained thus far. 
Let $\mathbb A$ be an associative algebra and $\mathfrak m$ be a Lie subalgebra of $\bba$. We may regard $\mathbb A$ as a module of $\mathfrak m$ via the adjoint action.

\begin{lemma}\label{l:A-inv}
Suppose that $\bba$ is locally finite under the adjoint action of $\fm$ and that $\fm$ is a reductive Lie algebra. Then
$\mathbb A=\mathbb A^{\mathfrak m}\oplus \ad(\fm)(\mathbb A)$.
\end{lemma}

\begin{proof}
Follows from the well-known result that, for any finite dimensional $\fm$ representation $(V,\pi)$ we have $V = V^\fm \oplus \pi(\fm)(V)$.
\end{proof}

We remark that the proof of Lemma \ref{l:A-inv} does not hold for the pair $(\cha\otimes\Cc,\fg)$ since finite dimensional $\fg=\mathfrak{spo}(2|2)$-modules are not completely reducible in general.

\smallskip

Now let $M$ be a simple $\mathbb A$-module and $\eta: \mathbb A\to \End_\bbc(M)$ the corresponding representation. The following result is the analogue in our setting of \cite[Lemma 4.2.3]{GW} or \cite[proof of Theorem 8, Lemma 1]{Ho}.

\begin{lemma}\label{l:finite}
Suppose $X$ is an $\mathfrak m$-invariant subspace of $M$ such that $X$ is a finite sum of simple lowest weight $\mathfrak m$-modules. Then the map $\mathbb A^{\mathfrak m}\to \Hom_{\mathfrak m}(X,M)$, $a\mapsto \eta(a)\mid_{X}$ is surjective.
\end{lemma}

\begin{proof}
Suppose $X=\sum_{i=1}^n L(\lambda_i)$, where $L(\lambda_i)$ are simple lowest weight module of $\mathfrak m$. Let $v_i$, $1\le i\le n$ denote the corresponding lowest weight vectors. Then every $0\neq \phi\in \Hom_{\mathfrak m}(X,M)$ is uniquely determined by the images $\phi(v_i)$, $1\le i\le n$. By Jacobson's Density Theorem, there is $a\in \mathbb A$ such that $\phi(v_i)=\eta(a)v_i$. This means that the natural map $\mathbb A\to \Hom_{\mathfrak m}(X,M)$ is surjective. Now,
let $B := \eta(\bba)|_X\subseteq \Hom_\bbc(X,M)$. From what was just said, we have $\Hom_{\mathfrak m}(X,M)\subseteq B$.
Since $B$ is quotient of $\bba$, it is also a locally finite $\ad(\fm)$-module, so just like in Lemma \ref{l:A-inv}, we have
$B = B^\fm\oplus\ad(\fm)(B)$. It follows that $\Hom_{\mathfrak m}(X,M)\subseteq B^\fm$. 
On the other hand, since the map $\mathbb A\to B$ is an $\ad(\mathfrak m)$-homomorphism, and from  Lemma \ref{l:A-inv},
we have that $\mathbb A^{\mathfrak m}\to B^{\mathfrak m}$ is surjective, yielding claim.
\end{proof}

The next result is the analogue of  \cite[Theorem 4.2.1]{GW}.

\begin{theorem}\label{t:Howe} Suppose that $M$ has a decomposition into $\mathfrak m$-isotypic components, $M=\bigoplus_{L}E(L) L,$ where $L$ ranges over a set of (isomorphism classes of) simple lowest weight modules of $\mathfrak m$, and such that the multiplicity spaces $E(L)=\Hom_{\mathfrak m}(L,M)$ are finite dimensional. Then:
\begin{enumerate}
\item[(i)] Every $E(L)$ is a simple $\mathbb A^{\mathfrak m}$-module.
\item[(ii)] If $E(L)\cong E(L')$ then $L=L'$.
\end{enumerate}
\end{theorem}

\begin{proof}
The proof is identical with the proof of \cite[Theorem 4.2.1]{GW} on page 196-197 of {\it loc.cit.}, using Lemma \ref{l:finite} in place of \cite[Lemma 4.2.3]{GW}.
\end{proof}

\begin{corollary}
The $(\Cu,\fs)$ decomposition in (\ref{e:sl2isotypic}) is multiplicity-free.
\end{corollary}

\begin{proof}
The pair $(\cha,\fs)$ satisfies the assumptions of Lemma \ref{l:A-inv}, as the commutator with $\mathfrak s=\langle H,E_1^+,E_1^-\rangle$ preserves the filtration subspaces $\cha^{(n)}$ defined by giving $x$ and $y$ degree $1$ and $\bbc W$ degree $0$,
which are finite-dimensional subspaces of $\cha$.
\end{proof}

\section{On the centralizer algebras in the classical case}\label{s:appendix}

In this section, let $\Cw$ denote the Weyl algebra acting on the polynomial algebra $\bbc[\fh]$ and let $\Cc$ denote the Clifford algebra of the pair $(V,B)$ as considered in Section \ref{s:cliff}.  For completeness, we provide arguments to determine the subalgebras of $\Cw$ and of $\Cw\otimes\Cc$
that centralizes the non-deformed versions of $\fs\fl(2,\bbc)$ and $\fs\fp\fo(2|2,\bbc)$. We shall
make use of the following two easy results. First, let $A$ be a filtered, unital associative $\bbc$-algebra. Denote by $A^{(t)}$, $t\in \bbn$, the pieces of the ascending filtration of $A$. Suppose further that $\Gamma,\fa$ are Lie (super)algebras contained in $A^{(2)}$ such that
\begin{itemize}
\item[(a)] $[\Gamma,\fa]=0$,
\item[(b)] if $P\in A^{(2)}$ commutes with $\Gamma$, then $P\in \fa \oplus \bbc$.
\end{itemize}
Denote by $A(\fa)$ the associative subalgebra of $A$ that is generated by $\fa$.  From (a), we have that $A(\fa)\subseteq \Cent_A(\Gamma)$. To determine whether $A(\fa)=\Cent_A(\Gamma)$, under these assumptions, consider:
\begin{lemma}\label{l:ind}
Let $t\in\bbz_{\geq 2}$ and suppose that  whenever $P\in A^{(t)}$ commutes with $\Gamma$ there exists $U\in A(\fa)$ such that
\begin{equation}\label{eq:topdeg}
P\equiv U \mod A^{(t-1)}.
\end{equation}
Then, $P\in A(\fa)$. In particular, if (\ref{eq:topdeg}) holds for all $t\geq 2$, then $A(\fa) = \Cent_A(\Gamma)$.
\end{lemma}

\begin{proof}
Let $U = U_1$. From (\ref{eq:topdeg}), we get that $P-U_1 \in A^{(t-1)}$. Since $P$ and $U_1$ commutes with $\Gamma$, we find $U_2\in A(\fa)$ such that $P - U_1 - U_2 \in A^{(t-2)}$. Proceeding by induction, we find $U_1,\ldots U_{t-2}\in A(\fa)$ such that $P - \sum_{j=1}^{t-2} U_j \in A^{(2)}$ and commutes with $\Gamma$. By (b), this is in $A(\fa)$. Since each $U_j\in A(\fa)$, then so is $P$. This finishes the proof.
\end{proof}

The next result is an observation, whose proof is clear.

\begin{lemma}\label{l:dirtytrick}
Let $E$ be the quotient of a vector space $U$ and $\sigma:E\to U$ a section of the quotient map $q:U\to E$. Given any finite set $\{e_1,\ldots,e_N\}$ such that $\{\sigma(e_1),\ldots,\sigma(e_N)\}$ is linearly independent in $U$, then $\sum_ia_ie_i = 0$ in $V$ implies $a_1=\cdots=a_N = 0$.
\end{lemma}

\subsection{The $\fs\fl(2)$-case}
Let $H,E_1^\pm$  and $\Cx_{ij},$ for $1\leq i < j \leq r$, be the elements in $\Cw$ 
analogous to the ones of $\cha$ considered in section \ref{s:sl2centralizer}. Then we know
that the Lie algebra $\fs$ spanned by $\{H,E_1^\pm\}$ is isomorphic to $\fs\fl(2,\bbc)$ 
and the Lie algebra $\fa$ spanned by $\{\Cx_{ij}\mid 1\leq i<j\leq r\}$ is isomorphic to $\fs\fo(r,\bbc)$.
In this section we will prove that the associative subalgebra $A(\fa)$ of $\Cw$ generated by $\fa$
is the centralizer algebra of $\fs = \fs\fl(2,\bbc)$.

With respect to the natural filtration in which elements of $V=\fh^*\oplus\fh$ have degree one,
the associated graded object $\overline\Cw$ of the Weyl algebra is isomorphic 
to $S(\fh^*)\otimes S(\fh)$, where $S(\fh^*)$ 
denotes the symmetric algebra on $\fh^*$.  If $X\in S^m(\fh^*)\otimes S^n(\fh)$, then $[H,X] = (m-n)X$. Hence, the eigenvalues of $\ad(H)$ on $\Cw$ are the integers $n\in\bbz$. Denote by $\Cw_n$ the $n$-eigenspace of $\ad(H)$. As a linear space, we thus have 
$\Cw_0 \cong \oplus _{p\in \bbn} S^p(\fh^*)\otimes S^p(\fh)$.

\begin{definition}
Let $\gls{Pairs}\subseteq\bbn^2$ be the set of all ordered pairs $\pi=(i,j)$ with $1\leq i,j \leq r$. We order $\Pi$ via $(i,j)\leq (k,l)$ under the lexicographic ordering of $\bbn^2$. Let $\tilde{\SM}_p$ 
denote the 
set of all multisets of cardinality $p$, that is, $\delta = \{\delta_1,\ldots,\delta_p\}$, with $\delta_k\in\Pi$ and we
allow an element of $\delta$ to have multiplicity. Let $\tilde{\SM} = \cup_p\tilde{\SM}_p$. 
\end{definition}

\begin{definition}
We say that $(i_1,j_1) \gls{orderM} (i_2,j_2)$ if and only if $i_1\leq i_2$ and $j_1\leq j_2$. Given $\mu=\{\mu_1,\ldots,\mu_p\}\in\tilde{\SM}$, we say that $\mu$ is a {\it double multi-index} if 
$\mu_k\ll\mu_{k+1}$, 
for all $1\leq k \leq p$. We let $\gls{Multiindices}\subseteq\tilde{\SM}$ denote the set of all double multi-indices and $\gls{Multiindicesp}$ be the subset of the ones of cardinality $p$. We shall write 
$\mu = [\mu_1,\ldots,\mu_p]\in\SM$ to indicate that elements of $\mu$ are ordered with respect to the partial ordering $\ll$.
\end{definition}

\begin{definition}
We order the set $\tilde{\SM}$ by declaring 
$\delta \gls{orderM2} \epsilon$ if $|\delta| < |\epsilon|$ and if $|\delta|=|\epsilon|$, then $\delta \preceq \epsilon$ lexicographically according to the total ordering of $\Pi$. The subset $\SM$ of double multi-indices is
thus endowed with this ordering.
\end{definition}

We shall now relate the notion of double multi-indices with linear bases of $\Cw_0$. Note first that 
the usual notion of multi-indices as vectors in $\bbn^r$ is equivalent to that of ordered
multisets $\alpha=[i_1,\ldots,i_p]$ with $1\leq i_k\leq i_{k+1}\leq r$ for all $k$. We shall write
$x_\alpha = x_{i_1}\cdots x_{i_p}$. Note that there is a bijective correspondence between elements in $\SM_p$ and pairs of multi-indices of
cardinality $p$, given by
\begin{equation}\label{eq:bijec}
\mu=[(i_1,j_1),\ldots,(i_p,j_p)] \mapsto ([i_1,\ldots,i_p],[j_1,\ldots,j_p])=(\alpha(\mu),\beta(\mu)).
\end{equation}
The inverse map will be denoted $(\alpha,\beta)\mapsto\mu(\alpha,\beta)$, obtained by setting $\mu_k=(i_k,j_k)$, for all $k$.

\begin{definition}
For each $(i,j)\in \Pi$, let $\gls{Matij} := x_iy_j\in\Cw_0$ and if $\mu\in\SM_p$ is a double multi-index put $\gls{Matmu} =.
\Cm_{\mu_1}\cdots \Cm_{\mu_p}\in\Cw_0$. Define also 
\[
\gls{Basisij} := \left\{
\begin{array}{rl}
\Cm_{ij}, &\textup{if } i\leq j\\
\Cx_{ij} = \Cm_{ij} - \Cm_{ji}, &\textup{if } i > j,
\end{array}
\right.
\]
and let $\gls{Basismu} = \Cb_{\mu_1}\cdots \Cb_{\mu_p}\in\Cw_0$, whenever $\mu\in \SM_p$.
\end{definition}

\begin{proposition}\label{p:CanLinBasis}
The set $\{\Cm_\mu\mid \mu\in \SM\}$ is a linear basis for the space $\Cw_0$. 
\end{proposition}

\begin{proof}
From the PBW property for $\Cw$, the set $\{x_\alpha y_\beta\mid \alpha,\beta \textup{ multi-indices with } |\alpha|=|\beta|\}$ is a linear basis of $\Cw_0$. 
Using the bijective correspondence of (\ref{eq:bijec}), let $\mu = \mu(\alpha,\beta).$ Since in the associated graded we have $\Cm_\mu = x_\alpha y_\beta$, we are done.
\end{proof}

\begin{definition}
Let $\Pi = \Pi'\cup \Pi''$ where $\Pi' = \{(i,j)\mid i > j\}$ and $\Pi''$ is the complement. We shall also write $\SM_p = \SM'_p\cup \SM''_p$ where $\SM_p'=\{\mu \in \SM_p \mid \mu_k\in \Pi' \textup{ for all } k\}$ and $\SM''_p$ is the complement. 
\end{definition}

\begin{proposition}\label{p:LinearBasis}
The set $\{\Cb_\mu\mid \mu\in \SM\}$ is also a linear basis for the space $\Cw_0$. 
\end{proposition}

\begin{proof}
Let $\mu\in\SM$. We claim that, with respect to the $\prec$ ordering, 
\begin{equation}\label{eq:BmuEmu}
\Cb_\mu = \Cm_\mu + \sum_{\nu \prec \mu} \sgn(\nu)\Cm_\nu,
\end{equation}
where the $\nu \prec\mu$ entering in the sum above are the ones obtained when we distribute the product over the factors $\Cx_{i_kj_k} = \Cm_{i_kj_k}-\Cm_{j_ki_k}$, which translates to 
transposing a pair $(i_k,j_k)$ to $(j_k,i_k)$ and reordering to get double multi-indices, if necessary. Granted the claim,
it follows that for each fixed $p$,  the linear span of $\{\Cm_\mu\mid \mu\in \cup_{q\leq p} \SM_q\}$ and $\{\Cb_\mu\mid \cup_{q \leq p} \SM_q\}$ inside $\Cw_0$ coincide. 
Moreover, the matrix of change of basis from $\{\Cm_\mu\mid \mu\in \cup_{q\leq p} \SM_q\}$ to $\{\Cb_\mu\mid \mu\in \cup_{q \leq p}\SM_q\}$ is triangular, because of (\ref{eq:BmuEmu}). 
So the claim will finish the proof.

As for the claim, suppose first that there is $k$ with $1 < k < p=|\mu|$ and $\mu_k=(i_k,j_k)$ is in 
$\Pi'$ (i.e., satisfies $i_k > j_k$). Let  $\bar\mu_k = (j_k,i_k)$ and
assume that $i_l\leq j_l$ if $l\in\{k-1,k+1\}$. Then,
\begin{equation}\label{eq:order}
i_{k-1}\leq j_{k-1}\leq j_k < i_k \leq i_{k+1}\leq j_{k+1}.
\end{equation}
Hence, $\nu = [\mu_1,\ldots,\mu_{k-1},\bar\mu_k,\mu_{k_1},\ldots,\mu_p]$ is also in $\SM_p$ (as (\ref{eq:order}) grants that $\mu_{k-1}\ll\bar\mu_k\ll\mu_{k+1}$) 
and $\nu \prec \mu$ since $\bar\mu_k < \mu_k$ in the lexicographic ordering.
Assume next that $\mu$ contains a chain
$\mu_{k-1}\ll\mu_k\ll \cdots \ll \mu_{k+q}\ll\mu_{k+q+1}$ satisfying $q>0$ and $\mu_l\in \Pi'$ in the range $k\leq l \leq k+q$. For simplicity, assume
$\mu_s\in \Pi''$ if $s<k$ or $s>k+q$. We are seeking to show that we can write
\[
\Cb_\mu = \Cm_{\mu_1}\cdots \Cm_{\mu_{k-1}}(\Cm_{\mu_k}-\Cm_{\bar\mu_k})\cdots 
(\Cm_{\mu_{k+q}}-\Cm_{\bar\mu_{k+q}})\Cm_{\mu_{k+q+1}}\cdots \Cm_{\mu_{p}} = 
\Cm_\mu + \sum_{\nu\prec \mu} \sgn(\nu)\Cm_\nu.
\]
So we need to consider the effect of the transpositions $\bar\mu_l$ in the product above. Note that for each $l$ in the range $k\leq l \leq k+q$ 
we get $ j_l < i_{l+1}$ and $j_{l-1}<i_l$, but we can not guarantee that $i_{l-1}\leq j_l$ or that $i_l\leq j_{l+1}$. Hence, the multi-set $\nu'=\{\mu_1,\ldots,\bar\mu_l,\ldots\mu_p\}$ 
may not be a double multi-index, since it may not be the case that $\mu_{l-1}\ll\bar\mu_l\ll\mu_{l+1}$. If $\nu'\in \SM$, the argument of the previous paragraph applies and we will be done. So suppose 
that $\nu'$ is not in $\SM$. We then need to reorder the chains
\begin{equation}\label{eq:chains}
i_1,\ldots,i_{l-1},j_l,i_{l+1},\ldots,i_p 
\quad\textup{ and }\quad j_1,\ldots,j_{l-1},i_l,j_{l+1}\ldots,j_p.
\end{equation}
When reordering the first chain, since $i_l>j_l$, then $j_l$ ends up in position $s$ with $s \leq l$ and $s$ the largest index such that $i_{s-1}\leq j_l<i_s$. It follows that 
$(j_l,j_s)<(i_s,j_s)$ and hence $\nu'\prec\mu$. The reordering of the second chain will yield a
double multi-index $\nu''\preceq\nu'\prec\mu$, as required. Such a reordering involves commutations in the Weyl algebra, so we may get multi-indices of same cardinality as $\mu$ and also multi-indices of smaller cardinality, which are necessarily smaller then $\mu$ in the $\prec$ ordering, as well. I

In the general case, a double multi-index may have several chains of pairs in $\Pi'$ and there may be  more than one transposition involved. In these cases, we can focus on the transposition of smallest index and apply the arguments above, since $\SM$ is lexicographically ordered. 
\end{proof}

\begin{example}
To illustrate (\ref{eq:BmuEmu}) in the argument of the previous proposition, consider the double multi-index $\mu = [(1,2),(4,2),(4,3),(4,4)]$. Then, 
$\Cb_\mu = \Cm_{12}(\Cm_{42}-\Cm_{24})(\Cm_{43}-\Cm_{34})\Cm_{44}$. Applying the
distribution yields
\[ \Cb_\mu = \Cm_\mu - \Cm_{\nu'_1} - \Cm_{\nu'_2} + \Cm_{\nu_3}. \]
Here, $\nu'_1 = \{(1,2),(2,4),(4,3),(4,4)\}, \nu'_2=\{(1,2),(4,2),(3,4),(4,4)\}$ and $\nu_3 = [(1,2),(2,4),(3,4),(4,4)]$. Note that $\nu'_1$ and $\nu'_2$ are not
double multi-indices since, respectively, $(2,4) \not\ll (4,3)$ and $(4,2) \not\ll(3,4)$. So we need to reorder these terms. But
\[\Cm_{\nu'_1} = (x_1y_2)(x_2y_4)(x_4y_3)(x_4y_4) =  (x_1y_2)(x_2y_3)(x_4y_4)^2 + (x_1y_2)(x_2y_3)(x_4y_4)\]
and
\[\Cm_{\nu'_2} = (x_1y_2)(x_4y_2)(x_3y_4)(x_4y_4) =  (x_1y_2)(x_3y_2)(x_4y_4)^2, \]
hence $\Cb_\mu = \Cm_\mu - (\Cm_{\nu_1}+\Cm_{\sigma_1}) - \Cm_{\nu_2} + \Cm_{\nu_3}$ with $\nu_1 = [(1,2),(2,3),(4,4)^{2}]$, $\sigma_1 =  [(1,2),(2,3),(4,4)]$, $\nu_2 = [(1,2),(3,2),(4,4)^{2}]$ and
$\nu_3$ as above. Moreover, $\nu_j\prec\mu$ for all $j$.
\end{example}

\begin{lemma}\label{l:sectionsl2}
The multiplication map $m: S^{q}(\fh^*\otimes\fh)\otimes S^2(\fh) \to S^{q}(\fh^*)\otimes 
S^{q+2}(\fh)$  is surjective and admits a section $\sigma$ that maps a monomial 
$x_{i_1}\cdots x_{i_q}\otimes y_{j_1}\cdots y_{j_{q+2}}$ into $\Cm_{\mu_1}\cdots\Cm_{\mu_q}\otimes y_{j_{q+1}}y_{j_{q+2}}$, with $\mu_k = (i_k,j_k)$.
\end{lemma}

\begin{proof}
For each $(i,j)\in\Pi$, we let $\tilde\Cm_{ij} = x_i\otimes y_j\in\fh^*\otimes\fh$. Then,
for each $\delta=\{\delta_1,\ldots,\delta_q\}\in\tilde\SM_q$ we let 
$\tilde\Cm_\delta = \tilde\Cm_{\delta_1}\cdots\tilde\Cm_{\delta_{q}}\in S^q(\fh^*\otimes\fh)$, so that the set $\{\tilde\Cm_\delta\otimes y_iy_j\mid\delta\in\tilde\SM_q, 1\leq i \leq j \leq r\}$ is a basis $S^q(\fh^*\otimes\fh)\otimes S^2(\fh)$. The multiplication map send each basis element to
$m(\tilde\Cm_\delta\otimes y_iy_j) = (x_{i_1}\cdots x_{i_q}\otimes y_{j_1}\cdots y_{j_{q}})(y_iy_j)$, from which
the claims follow.
\end{proof}

\begin{proposition}\label{p:centsl2}
The associative subalgebra $A(\fa)$ of $\Cw$ generated by $\fa\cong\fs\fo(r,\bbc)$ is the centralizer algebra of $\fs=\fs\fl(2,\bbc)$.
\end{proposition}

\begin{proof}
Let $t\geq 2$ and $P\in\Cw^{(t)}$. Suppose that $P$ commutes with  $\fs\cong\fs\fl(2,\bbc)$. Since it commutes with $H$, we have $P\in \Cw_0$, 
so in particular $t=2s$. Hence, we can write
\begin{equation}\label{e:P}
P = \sum_{\mu\in\SM_s} b_{\mu} \Cb_\mu + \xi
\end{equation}
with $\xi \in \Cw^{(t-1)}$. In order to check (\ref{eq:topdeg}), we need to show that 
$[E_1^-,P]=0$ implies $P$ is congruent to an element of $A(\fa)$, modulo $\Cw^{(t-1)}$. 
To that end, suppose that in the expression of $P$ in (\ref{e:P}) we only had double multi-indices
$\mu\in\SM_s''$. We shall show that, in this case, if $[E_1^-,P]=0$ then $P = 0$. This will imply the
result. Now, since 
$[E_1^-,\Cb_\mu] = \sum_k \Cb_{\mu_1}\cdots [E_1^-,\Cb_{\mu_k}]\cdots \Cb_{\mu_s}$, we obtain,
modulo $\Cw^{(t-1)}$, that
\[
[E_1^-,P] \equiv \sum_{\mu} \sum_{k\mid\mu_k\in\Pi''}
b_{\mu}\Cb_{\mu_1}\cdots \widehat{\Cb}_{\mu_k} \cdots \Cb_{\mu_s}y_{i_k}y_{j_k}.
\]
Recall that the associated graded object $\overline{\Cw}$ is isomorphic to $S(\fh^*\oplus\fh) = \oplus_{p,q \geq 0} S^p(\fh^*)\otimes S^q(\fh)$.  Further, we have that the top-degree piece of 
$[E_1^-,P]$ lives in the part of $\overline\Cw$ that is isomorphic to 
$S^{s-1}(\fh^*)\otimes S^{s+1}(\fh)$.
Note that the set $\{\Cb_\nu y_iy_j\mid\nu\in\SM_{s-1},1\leq i\leq j\leq r\}$ is not linearly independent in 
$S^{s-1}(\fh^*)\otimes S^{s+1}(\fh)$, due to commutation relations between $\Cb_\nu$ and $y_i,y_j$. However, $\{\tilde\Cb_\delta \otimes y_iy_j\mid\delta\in \tilde\SM_{s-1},1\leq i\leq j\leq r\}$
is a basis of $S^{s-1}(\fh^*\otimes\fh)\otimes S^2(\fh)$; the proof that 
$\{\tilde\Cb_\delta\mid \delta\in\tilde\SM_{s-1}\}$ is a basis of $S^{s-1}(\fh^*\otimes\fh)$ is similar
to the one discussed in Proposition \ref{p:CanLinBasis}. Hence, using Lemmas \ref{l:sectionsl2}
and \ref{l:dirtytrick}, we conclude that if $[E_1^-,P]=0$, then $b_\mu = 0$ for all $\mu$ since, if
we let $\sigma$ denote the section of the multiplication map as in Lemma \ref{l:sectionsl2}, then the
set $\{\sigma(\Cb_{\mu_1}\cdots \widehat{\Cb}_{\mu_k} \cdots \Cb_{\mu_s}y_{i_k}y_{j_k})\mid 
\mu\in\SM_s, k \textup{ with }\mu_k\in\Pi''\}$ is a subset of a linear basis of $S^{s-1}(\fh^*\otimes\fh)\otimes S^2(\fh)$
and thus linearly independent.
\end{proof}

\subsection{The $\fs\fp\fo(2|2)$-case}
We now let $H,E_1^\pm,E_2^\pm,E_3^\pm$  and $\Cx^\Delta_{ij}
=\Cx_{ij}\otimes 1 + 1\otimes\Cx_{ij},$ 
for $1\leq i < j \leq r$, be the elements 
analogous to the ones defined in (\ref{e:spodefs}) and (\ref{e:sodeformed1}) for $\Cw\otimes\Cc$.
The Lie super algebra $\fg$ spanned by $\{H,E_1^\pm,E_2^\pm,E_3^\pm\}$ 
is isomorphic to $\fs\fp\fo(2|2,\bbc)$ and the Lie algebra $\fa^\Delta$ spanned by
 $\{\Cx^\Delta_{ij} \mid 1\leq i<j\leq r\}$ is also isomorphic to $\fs\fo(r,\bbc)$. 
 We note that $\Cx^\Delta_{ij}$ is a diagonal copy of $\fs\fo(r,\bbc)$ in $\Cw\otimes\Cc$.
In this section we will prove that the associative subalgebra $A(\fa^\Delta)$ of $\Cw\otimes \Cc$ 
generated by $\fa^\Delta$ is the centralizer algebra of $\fg\cong \fs\fp\fo(2|2,\bbc)$.

First recall that with respect to the natural filtration in which elements of 
$V=\fh^*\oplus\fh$ have degree one, the associated graded object $\overline\Cc$ of the 
Clifford algebra is isomorphic to $\wedge(\fh^*)\otimes \wedge(\fh)$. Similarly to
the Weyl algebra, $\ad(Z)$ acts on $\wedge^m(\fh^*)\otimes \wedge^n(\fh)$
as $(m-n)$ times the identity, so we denote by $\Cc_n$ the $n$-eigenspace of $\ad(Z)$. 
We thus have  $\Cc_0 \cong \oplus _{p=0}^r \wedge^p(\fh^*)\otimes \wedge^p(\fh)$. Further,
let $A = \Cw\otimes\Cc$ and we endow it with the filtration 
$A^{(t)} = \sum_{p+q = t}\Cw^{(p)}\otimes\Cc^{(q)}$, from which the associated graded object satisfies
$\overline{A} = \overline{\Cw}\otimes\overline{\Cc}$.

For each $(i,j)\in \Pi$, let $\Cm_{ij} = x_iy_j\in\Cc_0$ and we define $\Cm_\mu$ with $\mu\in\SM$ just as before. However, since $x_j^2 = 0 = y_j^2$, the elements $\Cm_\mu$ with $\mu\in\SM$ 
may be zero in $\Cc$. 

\begin{definition}
Given $(i,j),(k,l)\in \Pi$, let us denote $(i,j)\gls{orderR}(k,l)$ if $i<k$ and $j<l$. Let $\gls{MultiindicesClp}\subseteq\SM_q$ be the subset of all double multi-indices of cardinality $q$
such that $\rho\in\SR_q$ if and only if $\rho_1\lll\rho_2\lll\cdots\lll\rho_q$. Let $\gls{MultiindicesCl} = \cup_{q=0}^{r}\SR_q$. 
\end{definition}

\begin{proposition}
The set $\{\Cm_\rho\mid\rho\in \SR\}$ is a linear basis of $\Cc_0$. 
\end{proposition}

\begin{proof}
Similar to Proposition \ref{p:LinearBasis}.
\end{proof}

We note that because the elements $\Cm_{ij}\in\Cc_0$ are quadratic, they actually commute in 
$\overline{\Cc}$. Thus, the linear span of $\Cm_\rho$, with $\rho\in\SR_q$ is isomorphic to
$\wedge^q(\fh^*)\otimes\wedge^q(\fh)$ which is a quotient of $S^q(\fh^*\otimes\fh)$, 
with linear basis $\{\tilde{\Cm_\mu}\mid\mu\in\tilde{\SM}_q$\}. Similar to the Weyl algebra case, 
we have natural sections in this setting as well. 

\begin{lemma}\label{l:sectionspo}
The multiplication map 
\[
m: S^{p}(\fh^*\otimes\fh)\otimes(\fh\otimes\fh^*)\otimes S^{q}(\fh^*\otimes\fh)\to 
(S^{p}(\fh^*)\otimes S^{p+1}(\fh)) \otimes \left(\wedge^{q+1}(\fh^*)\otimes \wedge^{q}(\fh)\right)
\]  
is surjective and admits a section $\sigma$ that maps 
\[
\sigma:x_\alpha \otimes y_{\beta}\otimes x_\xi\otimes y_\zeta\mapsto \tilde\Cm_{\mu_1}\cdots \tilde\Cm_{\mu_p}\otimes (y_{j_{p+1}}\otimes x_{m_{0}}) \otimes \tilde\Cm_{\nu_1}\cdots
\tilde\Cm_{\nu_q}
\]
with $\alpha=[i_1,\ldots i_p],\beta =[j_1,\ldots j_{p+1}]$, $\xi = [m_0,m_1,\ldots,m_q]$, 
$\zeta=[n_1,\ldots,n_q]$ and $\mu_k = (i_k,j_k)$, $\nu_k = (m_k,n_k)$.
\end{lemma}

\begin{proof}
This proof is analogous to Lemma \ref{l:sectionsl2}.
\end{proof}

\begin{proposition}\label{p:centspo}
The associative subalgebra $A(\fa^\Delta)$ of $A=\Cw\otimes\Cc$ generated by $\fa^\Delta$ is the centralizer algebra of $\fg\cong\fs\fp\fo(2|2,\bbc)$.
\end{proposition}

\begin{proof}
Let $t\geq 2$ and suppose that $P\in A^{(t)}$. Suppose that $P$ commutes with $\fg\cong \fs\fp\fo(2|2,\bbc)$. Since $P$ commutes with the even subalgebra of $\fg$, it follows that
$P\in A(\fa)\otimes \Cc_0$, from which $t = 2s$ is even. Put $q_0 = s$ if $s\leq r$ or $q_0=r$ if $s>r$. Thus, we can write $P \equiv  P_0 + P_1 + \cdots + P_{q_0}$ modulo $A^{(2s-1)}$ with 
$P_q \in \overline{\Cw}^{2p}\otimes\overline{\Cc}^{2q}$ and $p+q = s$. Further, we note that $P_0 = \sum_{\mu\in\SM_s'} c_\mu \Cx_\mu\otimes 1$. For each $\mu\in\SM_s'$, let
$\Cx^\Delta_\mu = \Cx^\Delta_{\mu_1}\cdots\Cx^\Delta_{\mu_s}$ 
so that the element $X=\sum_\mu c_\mu\Cx^\Delta_\mu$ has the same $\overline{\Cw}^{2s}\otimes\overline{\Cc}^{0}$ component as $P$. Thus, 
$P-X \equiv P'_1 + \cdots + P'_{q_0}$ commutes with $\fg$ and only have components in $\overline{\Cw}^{2p}\otimes\overline{\Cc}^{2q}$ with $q>0$. We claim that $P'_q = 0$, for all $q>0$. 
This implies that $P$ is congruent, modulo $A^{(t-1)}$, to an element of $A(\fa^\Delta)$, so the claim together with Lemma \ref{l:ind} yields the result.

For each $q>0$, let $p = s-q$ and write $P_q' = \sum_{\nu,\rho} c_{\nu,\rho}\Cx_\nu \otimes \Cm_\rho$, with the sum running over $(\nu,\rho)\in\SM_p'\times\SR_q$. Then, modulo $A^{(t-1)}$, we 
have $[E_2^+,P'_q] \equiv L_q + R_q$ with  
$L_q\in \overline{\Cw}^{2p-1}\otimes \overline{\Cc}^{2q+1}$ and $R_q\in \overline{\Cw}^{2p+1}\otimes \overline{\Cc}^{2q-1}$. Here, $L_q$ is obtained from 
$\sum_{\nu,\rho} c_{\nu,\rho}[E_2^+,\Cx_\mu](1\otimes\Cm_\rho)$ and $R_q$ from
$\sum_{\nu,\rho} c_{\nu,\rho}(\Cx_\mu\otimes 1)[E_2^+,\Ce_\mu]$. Thus,
\begin{equation}\label{e:R}
R_q = - \sum_{\nu,\rho} \sum_{k=1}^{q} c_{\nu,\rho} \Cx_\nu y_{j_k}\otimes x_{i_k}\Cm_{\rho_1}\cdots\widehat{\Cm}_{\rho_k}\cdots\Cm_{\rho_q}.
\end{equation}
Starting with $q=1$, we have $P'_1 = \sum_{\nu,(ij)} c_{\nu,(ij)}\Cx_\nu \otimes \Cm_{ij}$. The equation $[E_2^+,P-X]=0$ implies $R_1 = -\sum_{\nu,(ij)} c_{\nu,(ij)}\Cx_\nu y_j\otimes x_i = 0$, 
since $R_1$ is the only component of $[E_2^+,P-X]$ that lives in $\overline{\Cw}^{2s-1}\otimes \overline{\Cc}^{1}$.
In light of Lemma \ref{l:dirtytrick} we take
$E= S^{s-1}(\fh^*)\otimes S^{s}(\fh) \otimes \fh^*$ and $U = S^{s-1}(\fh^*\otimes\fh)\otimes(\fh\otimes\fh^*)$. Note that $R_1$ lies in the portion of 
$\overline\Cw^{2s-1}\otimes\overline{\Cc}^1$ that is isomorphic to $E$. Using the section 
of the surjective multiplication map $U\to E$ described in Lemma \ref{l:sectionspo} (with $p = s-1$ 
and $q=0$), we conclude that $R_1 = 0$ implies that all $c_{\nu,(ij)}=0$, since the set
$\{\sigma(\Cx_\nu y_j\otimes x_i)\mid \nu\in\SM_{s-1},1\leq i, j\leq r\}$ is part of a basis of $E$. Hence, $P'_1 = 0$. 

By induction, assuming $P'_u = 0$ for all $u<q$, the equation $[E_2^+,P-X]=0$ implies $R_q=0$, since from the inductive hypothesis, $R_q$ will be the only component of $[E_2^+,P-X]$ that lives in
$\overline{\Cw}^{2p+1}\otimes \overline{\Cc}^{2q-1}$, with $p+q = s$. Now let 
$E = S^{p}(\fh^*)\otimes S^{p+1}(\fh)\otimes\wedge^{q}(\fh^*)\otimes\wedge^{q-1}(\fh)$
and $U = S^{p}(\fh^*\otimes \fh)\otimes(\fh\otimes\fh^*)\otimes S^{q-1}(\fh^*\otimes \fh)$. 
Note that $R_q$ lies in the piece of $\overline{\Cw}^{2p+1}\otimes \overline{\Cc}^{2q-1}$
that is linearly isomorphic to $E$. Further the set
$\{\tilde{\Cb}_\nu\otimes (y_j \otimes x_i) \otimes \tilde{\Cm}_\tau\mid \nu\in\tilde{\SM}_p,\tau\in\tilde{\SM}_{q-1},1\leq i,j,\leq r\}$ is a basis of $U$, so from Lemmas \ref{l:dirtytrick}, \ref{l:sectionspo} and (\ref{e:R}), we conclude 
that $0=[E_2^+,P-X]$ forces all $c_{\nu,\rho}$ to be zero and hence $P'_q = 0$. Thus, $P-X = 0$, and we are done.
\end{proof}

\newpage
\pagestyle{myheadings}
\glssetwidest{xxxxxxxx}
\printnoidxglossary[type=main, style=TwoColListSymb, title=List of Symbols]

\end{document}